\documentclass[a4paper,11pt]{article}

\usepackage{stmaryrd}
\usepackage[T1]{fontenc}
\usepackage[utf8]{inputenc}
\usepackage{lipsum}
\usepackage{amsmath, amssymb, mathtools, amsthm}
\usepackage{comment}
\mathtoolsset{showonlyrefs} 
\usepackage{esint}
\usepackage{crossreftools}
\usepackage[usenames,dvipsnames]{xcolor}
\usepackage{graphicx}
\usepackage{subfigure}
\usepackage{minitoc}
\usepackage{tikz}
\usepackage{enumerate} 
\usepackage[margin=0.96 in]{geometry}
\usepackage{bbm} 
\usepackage[numbers,sort&compress]{natbib}
\usepackage[colorlinks=true]{hyperref}
\hypersetup{urlcolor=blue, citecolor=red}
\usepackage{mathtools}

\usepackage[title]{appendix}
\usepackage{dsfont}

\DeclarePairedDelimiter{\norm}{\lVert}{\rVert}

\DeclarePairedDelimiter{\prt}{(}{)}

\newcommand{\curlyE}{\mathcal{E}}

\newcommand{\curlyI}{\mathcal{I}}

\newcommand{\curlyL}{\mathcal{L}}

\newcommand{\curlyN}{\mathcal{N}}

\newcommand{\curlyP}{\mathcal{P}}

\newcommand \commentout[1] {}

\newcommand{\partialt}[1]{\dfrac{\partial#1}{\partial t}}

\DeclareMathAlphabet{\mathup}{OT1}{\familydefault}{m}{n}
\newcommand{\dx}[1]{\mathop{}\!\mathup{d} #1}

\theoremstyle{plain}
\newtheorem{theorem}{Theorem}[section]
\newtheorem{lemma}[theorem]{Lemma}
\newtheorem{proposition}[theorem]{Proposition}
\newtheorem{assumption}[theorem]{Assumption}
\newtheorem{corollary}[theorem]{Corollary}

\theoremstyle{remark}
\newtheorem{remark}[theorem]{\bf Remark}

\newcommand{\ie}{\emph{i.e.}}

\newcommand{\ddt}{\frac{\dx{}}{\dx{t}}}

\newcommand{\R}{\mathbb{R}}

\numberwithin{equation}{section}

\begin{document}

\title{Improved convergence rates for the Hele-Shaw limit in the presence of confining potentials}

\author{Noemi David\thanks{CNRS, Laboratoire de Mathématiques Raphaël Salem  Avenue de l'Université, BP.12, 76801, Saint-\'Etienne-du-Rouvray, France. Email address: noemi.david@univ-rouen.fr}\and{Alp\'ar R. M\'esz\'aros\thanks{Department of Mathematical Sciences, University of Durham, Durham DH1 3LE, United Kingdom. Email address: alpar.r.meszaros@durham.ac.uk}}\and{Filippo Santambrogio\thanks{
Universite Claude Bernard Lyon 1, CNRS, Ecole Centrale de Lyon, INSA Lyon, Université Jean Monnet, 
Institut Camille Jordan, UMR5208, 43 bd du 11 novembre 1918, 69622 Villeurbanne Cedex,
France. Email address: santambrogio@math.univ-lyon1.fr}}}

\maketitle
\begin{abstract}
Nowadays a vast literature is available on the Hele-Shaw or incompressible limit for nonlinear degenerate diffusion equations. This problem has attracted a lot of attention due to its applications to tissue growth and crowd motion modelling as it constitutes a way to link soft congestion (or compressible) models to hard congestion (or incompressible) descriptions. In this paper, we address the question of estimating the rate of this asymptotics in the presence of external drifts. In particular, we provide improved results in the 2-Wasserstein distance which are global in time thanks to the contractivity property that holds for strictly convex potentials.
\end{abstract}

\vskip .4cm
\begin{flushleft}
    \noindent{\makebox[1in]\hrulefill}
\end{flushleft}
\begin{flushright}
    \noindent{\makebox[1in]\hrulefill}
\end{flushright}
\vskip 1cm

\section{Introduction}

Nonlinear partial differential equations such as
\begin{equation}\label{eq: main}
    \partial_{t} \varrho_m = \Delta \varrho_m^m + \nabla \cdot (\varrho_m \nabla V), \qquad m>1,
\end{equation}
have been employed in a variety of applications such as, for instance, the description of pedestrian motion and tissue growth modelling. The density of individuals (or cells) at location $x\in\R^d$ at time $t \in [0,\infty)$ is denoted by $\varrho_m(x,t)$, and evolves under the effect of an external potential $V:\R^d\to\R$ and Darcy's law. Indeed, the porous medium term corresponds to $\nabla \cdot(\varrho_m\nabla p_m)$ where the pressure satisfies the law
\begin{equation}\label{eq: pressure power law}
    p_m= P_m(\varrho_m):=\frac{m}{m-1} \varrho_m^{m-1}, \qquad m>1.
\end{equation}
Later in the paper, we will also discuss the following pressure law
\begin{equation}
    \label{singular pressure}
        p_\varepsilon= P_\varepsilon(\varrho_\varepsilon):= \varepsilon \frac{\varrho_\varepsilon}{1-\varrho_\varepsilon}, \qquad \varepsilon>0.
\end{equation}
The velocity field of the equation is then given by $v = -\nabla p_m -\nabla V$.
It is well known (see for instance \cite[Theorem 11.2.8]{AGS2008}, \cite{Otto}) that under suitable assumptions on the data  equation~\eqref{eq: main} possesses a gradient flow structure with respect to the 2-Wasserstein distance associated to the energy
\begin{equation}
    \label{eq: energy}
    E_m(\varrho)= \int_{\R^d} \varrho V \dx x + \frac{1}{m-1}\int_{\R^d} \varrho^m \dx x.
\end{equation}
In this manuscript, when speaking about the solution to \eqref{eq: main}, we will mean the unique gradient flow given by \cite[Theorem 11.2.8]{AGS2008}.
It is clear that, at least formally, the functional $E_m$ should converge, as $m\to\infty$, to
\begin{equation}\label{eq: energy infty}
    E_\infty(\varrho)=
    \begin{dcases}
        \int_{\R^d} \varrho V \dx x, &\text{for } \varrho\leq 1 \quad\mathrm{a.e.},\\[0.3em]
        +\infty, &\text{otherwise.}
    \end{dcases}
\end{equation}
The gradient flow associated with the limit energy functional $E_\infty$ has been extensively analysed in the context of crowd motion description due to its connection to the model proposed in \cite{MRSV2011}, see also \cite{MRS2010}. Here the authors proposed a model in which rather than moving following the spontaneous velocity field, $\boldsymbol{U}$, (namely the velocity individuals would have in the absence of others), the crowd follows the projection of $\boldsymbol{U}$ onto the space of admissible velocities, $K$, which are those that preserve the constraint $\varrho\leq 1$ (for the detailed definition of $K$ we refer the reader to \cite{MRS2010, MRSV2011}). This constraint on the density is of fundamental importance in the modeling of crowd behaviour since, obviously, there is a maximum number of individuals that are allowed to occupy the same position in space. The model reads
\begin{equation}\label{eq: mrsv}
  \begin{dcases}
      \partialt{\varrho} - \nabla \cdot (\varrho \boldsymbol{u})=0,\\[0.5em]
      \boldsymbol{u}= P_K(\boldsymbol{U}).
  \end{dcases}  
\end{equation} 
Here, $P_K(\boldsymbol{U})$ stands for the projection operator onto the set $K$. In \cite{MRS2010}, the authors show that when the velocity field has a gradient structure, namely $\boldsymbol{U}=-\nabla V$, system \eqref{eq: mrsv} can be seen as the gradient flow associated to $E_\infty$. When the problem is set in a bounded domain rather than the whole space, a classical example of potential $V$ is the distance to a door (represented by a flat bounded region of the domain's boundary). The constraint on the density can be clearly noticed in the definition of $E_\infty$ by the fact that the energy blows up when $\varrho>1$. We also mention the related works \cite{MS2016,DM2016}, which also consider similar models with non-gradient drift, in the presence of an additional linear diffusion.

However, the approximation of this problem with the porous medium equation \eqref{eq: main} was only suggested in \cite{MRSV2011, MRS2010}. This open question was addressed in \cite{AKY14} where the authors show convergence of the solution to \eqref{eq: main} to the gradient flow associated to $E_\infty$, $\varrho_\infty$, in the 2-Wasserstein, distance as $m\to\infty$. The main goal of \cite{AKY14} is to show how the porous medium equation \eqref{eq: main} is actually related to the following free boundary problem of Hele-Shaw type
\begin{equation}\label{HS}
    \begin{cases}
        -\Delta p = \Delta V, &\text{in } \Omega(t):=\{p>0\},\\[0.5em]
        v_\nu= -\partial_\nu p -\partial_\nu V,  &\text{on } \partial\Omega(t),
    \end{cases}
\end{equation}
where $v_\nu$ is the normal velocity of the free boundary. Since it is not clear how to derive geometrical properties of the limit $\varrho_\infty$, together with the convergence in the 2-Wasserstein distance the authors show uniform convergence of $\varrho_m$ to the solution of \eqref{HS} in the case in which the initial data is a patch, namely $\varrho_0=\mathds{1}_{\Omega_0}$ and the potential $V$ is semi-convex with uniformly bounded and strictly positive $\Delta V$. Consequently, the two limits coincide almost everywhere and, in this case, we have $\varrho_\infty=\mathds{1}_{\Omega(t)}$.
  
Let us stress that this link to the geometrical problem \eqref{HS} already emerged in \cite{MRSV2011}. Indeed, the authors show that the velocity field $\boldsymbol{u}=P_K(\boldsymbol{U})$ is actually given by $\boldsymbol{u}=\boldsymbol{U}-\nabla p$, where the pressure $p$ satisfies the \textit{complementarity relation}
$$\int_{\R^d} \boldsymbol{u}(x,t)\cdot \nabla p(x,t) \dx x=0, \text{  for a.e. } t>0,$$
which, for gradient spontaneous velocity fields, $\boldsymbol{U}=-\nabla V$, can be rewritten as
$$\int_{\R^d} p \Delta (p+V)\dx x=0.$$
The rigorous proof of this relation in a distributional framework constitutes a question that has been largely addressed in the literature, especially in relation to tissue growth models. The first result is due to \cite{PQV}, where the authors study the following equation
\begin{equation}\label{pqv}
    \partialt{\varrho_m} = \Delta \varrho_m^m + \varrho_m G(p_m),
\end{equation}
where the pressure is given by \eqref{eq: pressure power law} and the pressure-dependent growth rate $G$ is a decreasing function representing the fact that proliferation decreases in regions with higher congestion. 
This result was later extended to a variety of tissue growth models involving different pressure laws, \cite{hecht2017incompressible, DHV}, multiple species \cite{BPPS, LX2021}, and Brinkman's law \cite{PV, DeSc, DEBIEC2020}. The case in which, besides reactions, cells also undergo an external drift has been addressed in \cite{KPW19, DS21}. In particular, in \cite{DS21} the authors study the limit for weak solutions, proving the complementarity relation also in the presence of reactions, $\varrho_m G(p_m)$. 
In this case, the weak limit $\varrho_\infty$ satisfies
\begin{equation}\label{eq: limit}
\left\{
\begin{aligned}
    \partialt{\varrho_\infty} &=\Delta p_\infty +\nabla \cdot(\varrho_\infty\nabla V) + G(p_\infty),\\[0.5em]
    0&=p_\infty(1-\varrho_\infty).
    \end{aligned}
\right.
\end{equation}
It is straightforward to notice the analogy to \eqref{HS} for $G=0$. Indeed, for $\varrho_\infty<1$ the pressure term disappears and the equation is simply of transport type with velocity field $-\nabla V$, while for $p_\infty>0$ (hence $\varrho_\infty=1$) one formally recovers the complementarity relation.

In many applications to cell movement, models usually include interaction potentials describing nonlocal effects in addition to the local repulsion given by the porous medium term
\begin{equation}\label{eq: nonlocal}
    \partialt{\varrho_m} = \Delta \varrho_m^m +\nabla\cdot(\varrho_m \nabla W\star \varrho_m).
\end{equation}
The energies associated with the gradient flows of \eqref{eq: nonlocal} and the corresponding limit equation are
\begin{equation}
    F_m(\varrho) = \frac 12 \int_{\R^d} \varrho W\star \varrho \dx x + \frac{1}{m-1}\int_{\R^d} \varrho^m \dx x,
\end{equation}
and 
\begin{equation}
    F_\infty(\varrho) =
    \begin{dcases}
        \frac 12 \int_{\R^d} \varrho W\star \varrho \dx x, &\text{ for } \varrho\leq 1,\\[0.5em]
        +\infty &\text{ otherwise}.
    \end{dcases}
\end{equation}
In \cite{CKY18} the authors study the limit $m\to\infty$ of solutions to equation~\eqref{eq: nonlocal} for $W=\curlyN$, the Newtonian potential. Later, in \cite{perthamePKS}, the authors address the same problem for a Patlak-Keller-Segel tumour growth model including an additional reaction term. We also refer the reader to \cite{CT2020, CG2021, CCY2019} and references therein.

\medskip

Although the incompressible limit has proven to be widely studied and employed in various applications, it is interesting to notice how in the literature only a few results on the convergence rate can be found.
It is the goal of this paper is to contribute to the investigation of this question by estimating the rate of convergence of the distance $W_2(\varrho_m, \varrho_\infty)$ as $m \to \infty$.

\subsection{Previous results on convergence rates}

The first known result for the convergence rate of solutions to equation~\eqref{eq: main} as $m\to\infty$ is due to~\cite{AKY14}. The authors exploit the fact that the solutions can be approximated by using a minimizing movement scheme, the celebrated Jordan--Kinderlehrer--Otto (JKO) scheme. Then, they compute the convergence rate for the discrete in time approximations. This translates into a polynomial rate for the continuous solution, \cite[Theorem~4.2.]{AKY14}, that is 
 \begin{equation*}
     \sup_{t\in[0,T]} W_2(\varrho_m(t),\varrho_\infty(t))\leq \frac{C(T)}{m^{1/24}},
 \end{equation*}
where $C$ is a positive constant depending on $\int_{\R^d} V \varrho^0 \dx x$, $\|\Delta V\|_\infty$ and $T$.
The result is obtained under the following assumptions on the potential: there exists $\lambda\in\R$ such that
\begin{equation*}
      D^2 V(x) \geq \lambda I_d, \quad \forall x \in \R^d, \qquad   \inf_{x\in\R^d} V(x) = 0, \qquad
       \|\Delta V\|_{L^{\infty}(\R^d)}\leq C.
\end{equation*}

A later result was established in \cite[Theorem~1.1]{DDP2021} and gives a much faster polynomial rate in the $\dot{H}^{-1}$-norm
   \begin{equation} 
        \sup_{t\in[0,T]} \|\varrho_m(t) - \varrho_\infty(t)\|_{\dot{H}^{-1}(\R^d)} \leq\frac{C(T)}{{m}^{1/2}}.
    \end{equation}    
We underline that it is in general possible to include in the estimate the possibility to have different initial data for the equation with power $m$ and the limit equation, and that indeed the result in \cite{DDP2021} includes this possibility and more precisely provides the following estimate
 \begin{equation} 
        \sup_{t\in[0,T]} \|\varrho_m(t) - \varrho_\infty(t)\|_{\dot{H}^{-1}(\R^d)} \leq\frac{C(T)}{{m}^{1/2}} + \norm{\varrho_{m,0}-\varrho_{0}}_{\dot{H}^{-1}(\R^d)}.
    \end{equation}   
Even though equation~\eqref{eq: main} is not a gradient flow with respect to the $\dot{H}^{-1}$-norm, this choice allows the authors to account for linear reaction terms, $\varrho G(x,t)$. Since conservation of mass does not hold, in this case, it would not be possible to employ the classical $2$-Wasserstein distance. The strategy of \cite{DDP2021} relies on computing the differential equation satisfied by $\|\varrho_m(t)-\varrho_\infty(t)\|_{\dot{H}^{-1}}$ and using Gr\"onwall's lemma upon controlling the nonlinear diffusion and convective terms using integration by parts and Sobolev's and Young's inequalities. The result of \cite{DDP2021} is valid whenever $V$ has bounded second derivatives. More precisely, the constant $C(T)$ depends on the lower bound $\lambda\in\R$ such that the potential satisfies
\begin{equation}\label{assum: H-1}
    D^2V - \frac{\Delta V}{2}I_d \geq \lambda I_d.
\end{equation}
Note that this can be obtained when we have lower bounds on $D^2V$ and upper bounds on $\Delta V$.

Moreover, the authors assume that the equation is equipped with non-negative initial data $\varrho_{m,0}\geq 0$ such that there exists a compact set $K\subset\R^d$ and a function $\varrho_0\in L^1(\R^d)$ satisfying 
\begin{equation}\label{assum: DDP}
    {\rm{spt}}({\varrho_{m,0}})\subset K,\quad p_{m,0}=P_m(\varrho_{m,0})\in L^\infty(\R^d),    \quad \varrho_{m,0}\in L^1(\R^d),  \quad\norm{\varrho_{m,0} - \varrho_{0}}_{L^1(\R^d)}\rightarrow 0.
\end{equation}
The assumption on the compact support of the initial data is needed in order to ensure that the pressure satisfies an $L^\infty$ uniform bound, namely, there exists a positive constant $p_M=p_M(T)$ such that 
\begin{equation}
    \label{bound pressure}
    0 \le p_m \leq p_M, \qquad \text{for all } t\in [0,T].
\end{equation} 
We refer the reader to \cite[Lemma~A.10]{KPW19} for the local in time $L^\infty$ uniform bound on the pressure for an equation including reaction terms. Let us point out that if the reaction rate is pressure dependent as in \eqref{pqv}, $G=G(p)$, under suitable assumptions on $G$ the bound \eqref{bound pressure} can be proven to be global in time.

In \cite{DDP2021}, the authors also treat the case of the so-called singular pressure \eqref{singular pressure}.
This pressure law is frequently used in the modelling of crowd motion or tissue growth, see for instance \cite{hecht2017incompressible, BDDR2008, DHV, DegondHua2013}, as it has the crucial advantage of accounting for the fact that already at the level $\varepsilon>0$ the density cannot overcome a certain threshold. Indeed, the solution, $\varrho_\varepsilon$ of equation~\eqref{eq: main} with pressure law \eqref{singular pressure} always satisfies $\varrho_\varepsilon<1$. As $\varepsilon\to 0$ the equation also converges to equation~\eqref{eq: limit}, therefore this singular limit is also referred to as incompressible limit. In \cite{DDP2021} the authors establish a rate of at least $\sqrt\varepsilon$ in the $H^{-1}$ norm.

Finally, in \cite{CKY18} the authors find an explicit polynomial rate of convergence as $m\to\infty$ for a porous medium equation with Newtonian interaction of the order $m^{-1/144}$, using the JKO scheme as done in \cite{AKY14} for local drifts.

\subsection{Summary of the strategy and main contributions}

The main contribution of this paper is to provide a new result on the convergence of $\varrho_m$ to $\varrho_\infty$ as $m\to\infty$ in the 2-Wasserstein distance. Unlike \cite{AKY14}, our strategy does not rely on employing a time discretization but is rather based on the same idea used in \cite{DDP2021} for the $\dot{H}^{-1}$-norm -- we compute the time derivative of the square of the distance between $\varrho_{m_1}$ and $\varrho_{m_2}$ (solutions of equation \eqref{eq: main} for $1<m_1<m_2$), we exploit the obtained differential inequality using Gr\"onwall's lemma, and then we let $m_2\to \infty$. Computing this time derivative is a well-known tool in optimal transport theory, and it is a powerful argument to prove the uniqueness of solutions, as also applied in \cite{DM2016} to prove the uniqueness of solutions to the limit Hele-Shaw problem. We refer also to the more classical results \cite{CMV2006, BGG2012, BGG2013} where similar in spirit computations have been exploited to show convergence to equilibrium of various drift-diffusion models. The novelty here is that $\varrho_{m_1}$ and $\varrho_{m_2}$ do not satisfy the same equation, and therefore we need finer arguments to estimate the contributions coming from the nonlinear diffusion part of the equation, which will yield the polynomial rate.

This method allows us to obtain a better result for the $2$-Wasserstein distance. Moreover, the fact that for convex potentials the 2-Wasserstein distance is contractive allows us to infer a convergence result that is \textit{global in time} which constitutes the main novelty of this paper. 
We also account for interaction potentials, equation \eqref{eq: nonlocal}, and we actually present the result in a unified way, hence for an equation which includes both effects (see equation \eqref{main tot} below)
where the pressure $p$ is given either by \eqref{eq: pressure power law} or \eqref{singular pressure}.
A joint convexity condition on $V+W$ will yield the result globally in time.

\smallskip

\underline{\smash{\textit{Rate in $W_2$.}}}
We first estimate the convergence rate of $\varrho_\varepsilon$ solution of \eqref{main tot} as $\varepsilon\to 0$ in $W_2$ for the singular pressure law \eqref{singular pressure}. Let us stress that in \cite{DDP2021} the authors do not need any assumption on the support of the initial data for this kind of pressure law since it already implies a uniform bound $\varrho_\varepsilon <1$ for all $\varepsilon>0$. Here we observe that this property implies that rather than computing the time derivative of $W_2(\varrho_{\varepsilon_1}, \varrho_{\varepsilon_2})$, for $\varepsilon_1<\varepsilon_2$ and taking $\varepsilon_1\to 0$, we can directly compute the time derivative of $W_2(\varrho_\varepsilon, \varrho_\infty)$, where $\varrho_\infty$ is the gradient flow solution associated to the limit energy. We obtain the same polynomial rate of $\sqrt{\varepsilon}$.
Also for the porous medium case, namely \eqref{main tot} with \eqref{eq: pressure power law}, we obtain a polynomial rate of $1/\sqrt m$.
Let us stress that even if the rate is the same as the one found for the $\dot{H}^{-1}$-norm in \cite{DDP2021}, our strategy is independent of the inequalities which exist between the two distances. Indeed, let us recall that the negative Sobolev norm can be bounded by the $W_2$ distance when densities are bounded from above and that the converse inequality is true when densities are bounded from below by a strictly positive constant (see \cite{P2018} or \cite[Section 5.5.2]{OTAM}). Since a common lower bound away from zero is obviously not attainable on $\R^d$ for measures with finite mass, it is impossible to deduce an estimate in terms of $W_2$ from that in terms of $\dot{H}^{-1}$. On the other hand, an upper bound could be satisfied, locally in time, under suitable conditions on the potentials $V, W$, but this only shows that the estimate in the present paper is stronger than that in \cite{DDP2021}.

\smallskip

\underline{\smash{\textit{Relaxed assumptions on the initial data.}}}
Let us point out that we also introduce a technical change with respect to the proof of \cite[Theorem~1.1]{DDP2021}, specifically, we perform a different treatment of the porous medium part of the equation. To deal with this term, in \cite{DDP2021} the authors imposed assumptions~\eqref{assum: DDP} on the initial data in order to control the pressure uniformly in $m$, see bound~\eqref{bound pressure}. We are able to relax these assumptions, in particular, considering initial data that are not compactly supported, since our argument does not rely on any $L^\infty$ control of $p_m$, but rather on proving the integrability of $\varrho_m^{2m-1}$. Let us stress, however, that in the absence of reaction, as is the case for equation~\eqref{eq: main}, it would still be possible to infer $p_\infty \in L^\infty(0,\infty; L^\infty(\R^d))$ without any assumption on the support of $\varrho_{m,0}$, but rather asking a control on $\max_x V - \min_x V$. However being interested in accounting for convex potentials (such as $V(x)=|x|^2$), on the whole space $\R^d$, this assumption would not be suitable in our context, and moreover such an estimate holds for $m=\infty$ but in general cannot be obtained on $p_m$ for finite $m$ in a way which is uniform in $m$.

\smallskip

\underline{\smash{\textit{Rate in $L^1$ and $W_2$ for stationary states.}}}
The fact that we obtain an estimate that is global in time easily implies that the same rate of convergence holds for stationary states. The very same rate can be computed by using the shape of such stationary states by exploiting some geometric properties of the 2-Wasserstein distance. Indeed, first we compute a convergence rate in $L^1$ for the stationary states of equation~\eqref{eq: main} and then we translate it into a $W_2$ rate. Moreover, under certain conditions on the confining potential $V$, we show that for stationary solutions the rate in $W_2$ is actually faster than $1/\sqrt{m}$. This leaves the question of whether our convergence rate is sharp or not for the evolution problem open.
 
Let us emphasise that using $W_2$ yields a global-in-time result under more natural assumptions than those that would be required in the $\dot H^{-1}$-norm.  This is linked to the fact that, unlike for the $W_2$ distance, the drift part of equation~\eqref{eq: main} is not a gradient flow with respect to the $H^{-1}$ norm. In order to have a global rate in $H^{-1}$, one would need to have $\lambda>0$ in \eqref{assum: H-1}, and this does not hold for convex potentials which represent, from the point of view of applications, the most relevant cases. Indeed, when $V$ is convex the attractive nature of the potential ``competes'' with the repulsion given by the porous medium term.  Moreover, in dimension $d=2$, condition \eqref{assum: H-1} can never be satisfied for $\lambda>0$. Although for $d\geq 3$ there exist potentials for which the condition holds -- for instance potentials for which all the eigenvalues of the Hessian matrix are negative -- we remark that we would also need the assumption $|\nabla V|^2 \leq C V$ to have the global result, which is not compatible with concave potentials. This condition is required to ensure that the energy is controlled and that we can bound $\varrho_m^{2m-1}$ uniformly in $L^1$ as shown in Proposition~\ref{prop: bounds of powers}. We conclude that $W_2$ allows to treat the most interesting cases and yields additional properties thanks to the gradient flow structure of the equation.
It remains an open question to investigate whether it is possible to establish a global-in-time convergence rate in the $\dot{H}^{-1}$-norm for convex potentials. We would like to emphasize that the rate of convergence in this norm would be a direct consequence of the convergence rate in $W_{2}$, as long as one would be able to guarantee uniform upper bounds on the densities, globally in time. This, however, seems to be an open question for unbounded potentials. It is worth mentioning that these uniform upper bounds can be guaranteed locally in time, and so, these will give the desired convergence rates as well.

\subsection{Structure of the rest of the paper} In the following section we state the assumptions and the main results. Section~\ref{sec: prelim} is devoted to recalling the definition and properties of the 2-Wasserstein distance and to proving some preliminary results. Section~\ref{sec: global} contains the proofs of the main result, Theorem~\ref{thm: main}. The proofs of the results concerning the the stationary states, Theorem~\ref{thm: L1 s.s.} and Proposition~\ref{prop: s.s. improved}, are the object of Section \ref{sec:5}.

\section{Assumptions and main results}

Here we state the main results of the paper concerning the rate of convergence in the $2$-Wasserstein distance for local and non-local drifts and the improved results for the stationary states.

\subsection{Results in the 2-Wasserstein distance}

We will consider the following equation involving both local and nonlocal drifts
\begin{equation}\label{main tot}
    \partialt{\varrho} = \nabla\cdot(\varrho (\nabla p+ \nabla V + \nabla W \star \varrho))
\end{equation}
and the corresponding limiting problem
\begin{equation}\label{main:limit}
\left\{
\begin{aligned}
    \partial_{t}\varrho &= \Delta p_{\infty} + \nabla\cdot(\varrho_{\infty} ( \nabla V + \nabla W \star \varrho_{\infty}))\\[5pt]
    0&=p_{\infty}(1-\rho_{\infty}).
\end{aligned}
\right.    
\end{equation}

\begin{assumption}[Assumptions on the potentials]\label{assumptions potentials}
  Let $V:\R^d\to\R_+$, $W:\R^d\to\R_+$, such that $W(x)=W(-x)$ and 
\begin{align}\label{eq: assum V}
     \alpha I &\leq D^2 V \leq A I,  \; \text{for some } \: \alpha, A \in \R\\[0.5em] 
     \label{eq: assume W}
    \beta I &\leq D^2 W \leq B I, \; \text{for some } \; \beta, B \in \R.
\end{align}
\end{assumption}
\begin{assumption}[Assumptions for global result]\label{assum: global}
  Let $V:\R^d\to\R_+$, $W:\R^d\to\R_+$ satisfy Assumption~\ref{assumptions potentials}, and let $\alpha, \beta$ satisfy either
\begin{itemize}
 \item[(i)] $\alpha >0$ and $\alpha + \beta >0$, or
    \item[(ii)] $V=0$ and $\beta>0$.
\end{itemize} 
\end{assumption}

\medskip

\begin{remark}[Conservation of the center of mass.]\label{rmk: center mass}
Let us recall that, for $V=0$, equation~\eqref{main tot} preserves the center of mass. Indeed, the equation can be written in the form
\begin{equation}
    \partialt \varrho = \Delta Q(\varrho) + \nabla\cdot(\varrho \nabla W \star \varrho),
\end{equation}
where the function $Q:[0,+\infty)\to \R$ is either $Q(\varrho)=\varrho^m$ in the porous medium case, or $Q(\varrho)=\varepsilon H(\varrho)$ for the singular pressure law, with $H$ defined in \eqref{eq: H}. By integration by parts, we have
\begin{align}
    \ddt \int_{\R^d} x \varrho \dx x = \int_{\R^d} x \Delta Q(\varrho) \dx x+ \int_{\R^d} x \nabla \cdot(\varrho\nabla W\star \varrho)\dx x=-d\int_{\R^d} \varrho\nabla W\star \varrho\dx x.
\end{align}
Since by assumption $W$ is even, the last integral is also equal to zero by the following computation
    \begin{equation}
\iint_{\R^d\times\R^d} \varrho(x)\nabla W(x-y) \varrho(y)\dx x  \dx y =  - \iint_{\R^d\times\R^d} \varrho(x)\nabla W(y-x) \varrho(y)\dx x  \dx y  = -\int_{\R^d} \varrho \nabla W \star \varrho \dx y, 
    \end{equation}
    hence $\int_{\R^d} x \varrho(t)\dx x=\int_{\R^d} x \varrho_0\dx x$ for almost every $t>0$. Let us notice that the preservation of the center of mass is independent of the parameters $m>1$ and $\varepsilon>0$ in the pressure laws \eqref{eq: pressure power law} and~\eqref{singular pressure}.
    \end{remark}

\smallskip

\noindent
We now state the main result of the paper. For the sake of simplicity, we write one statement for both the case of singular pressure law \eqref{singular pressure} and power law \eqref{eq: pressure power law}. Therefore, we indicate the solution, $\varrho_m$, with the same index and in the statement $1/m=\varepsilon.$ 
 
\begin{theorem}[Rate of convergence in $W_2$]\label{thm: main}
Let $\varrho_m$ be the solution of \eqref{main tot} coupled with either \eqref{eq: pressure power law} or \eqref{singular pressure} and endowed with initial data $\varrho_0\in \curlyP_2(\R^d), \int_{\R^d} x \varrho_0 \dx x=0$, $\|\varrho_0\|_\infty \leq 1$. For all $T>0$, under Assumption~\ref{assumptions potentials} there exists $\varrho_\infty\in C(0,T; \curlyP_2(\R^d))$ such that $\varrho_m(\cdot, t)$ converges to $\varrho_\infty(\cdot, t)$ in the 2-Wasserstein distance as $m\to\infty$ uniformly in time with the following convergence rate
    \begin{equation}
     \sup_{t\in[0,T]} W_2(\varrho_m(t), \varrho_\infty(t))  \leq \frac{C(T)}{\sqrt{m}},
\end{equation}
where $C$ is a uniform positive constant depending on the final time $T$. 
Moreover, if Assumption~\ref{assum: global} holds, then the result holds globally in time, hence 
there exists $\varrho_\infty\in C(0,\infty; \curlyP_2(\R^d))$ and $C>0$ such that
    \begin{equation}
     \sup_{t\in[0,\infty)} W_2(\varrho_m(t), \varrho_\infty(t))  \leq \frac{C}{\sqrt{m}}.
\end{equation}
In both cases, there exist the corresponding pressure functions $p_{\infty}:(0,T)\times\R^{d}\to [0,\infty)$, and $p_{\infty}:(0,\infty)\times\R^{d}\to[0,\infty)$, respectively, such that the pair $(\varrho_{\infty},p_{\infty})$ is the unique weak solution to \eqref{main:limit} with initial condition $\varrho_{0}$.

\end{theorem}

\begin{remark}[General sequence of initial data]
Let us notice that, for the global result of Theorem~\ref{thm: main}, in the porous medium case, the condition $\|\varrho_0\|_\infty\leq 1$ can be relaxed. We may take any sequence of initial data $\varrho_{0,m}\in \curlyP_2(\R^d)\cap L^1(\R^d)$, such that there exists $\varrho_{0}\in \curlyP_2(\R^d), \varrho_{0}\leq 1$, and
\begin{equation*}
    W_2(\varrho_{0,m}, \varrho_{0}) \to 0, \text{ as } m\to\infty.
\end{equation*}
In this case, if $\varrho_m$ is the solution of \eqref{main tot} with initial data $\varrho_{0,m}$ and $\varrho_\infty$ is the limit flow with initial data $\varrho_0$, by the triangular inequality between $\varrho_m, \varrho_\infty,$ and the solution $\tilde\varrho_m$ of equation~\eqref{main tot} with initial data $\tilde\varrho_m(\cdot,0)=\varrho_{0},$ we have
   \begin{align}
    \sup_{t\in[0,\infty]} W_2(\varrho_m(t), \varrho_\infty(t)) &\leq \sup_{t\in[0,\infty]} (W_2(\tilde\varrho_m(t), \varrho_\infty(t))+ W_2(\varrho_m(t), \tilde\varrho_m(t)))\\[0.8em]
    &\leq \frac{C}{\sqrt{m}} +  W_2(\varrho_{0,m}, \varrho_{0}),
\end{align}
where we used the contractivity in the $2$-Wasserstein distance of each porous medium equation for any $m$ and Theorem~\ref{thm: main}. Note that, in general, for the porous medium equation with a semiconvex potential $V$ one has $W_2(\varrho_m(t),\tilde\varrho_m(t))\leq e^{\lambda t}W_2(\varrho_m(0),\tilde\varrho_m(0))$, with an exponentially growing factor, but the assumptions guaranteeing the validity of Theorem \ref{thm: main} also imply that we can replace this factor by the constant $1$.
Hence the rate is the worst between the rate of convergence of the initial data and $1/\sqrt{m}$.
\end{remark}

\subsection{Results for the stationary states}

Let us now consider the equation with only a local potential, equation \eqref{eq: main}. As already mentioned in the introduction, it is well known that under suitable assumptions on $V$, in particular for convex potentials, the solution to equation~\eqref{eq: main} converges exponentially to the {unique} stationary state as $t\to\infty$. For $m>1$ the global minimizer of $E_m(\varrho)$ has the following form
\begin{equation}\label{eq: s.s. m}
    \bar\varrho_m(x) = \prt*{\frac{m-1}{m}(C_m - V(x))_+}^{\frac{1}{m-1}}, 
\end{equation}
where $C_m$ is a positive constant such that
$$\int_{\R^d} \prt*{\frac{m-1}{m}(C_m - V(x))_+}^{\frac{1}{m-1}} \dx x =1,$$
while for $m=\infty$ the stationary state is the characteristic function
\begin{equation}\label{eq: s.s. infty}
 \bar\varrho_\infty(x)=\mathds{1}_{\{C_\infty > V(x)\}},
\end{equation}
where $C_\infty=\lim_{m\to\infty}C_m$ and the measure of the set $\{x\in\R^d: \ C_\infty>V(x)\}$ is equal to 1.

\begin{theorem}[Polynomial {convergence} rate in the $L^1$-norm for the stationary states]\label{thm: L1 s.s.}
Let $\bar\varrho_m, \bar\varrho_\infty$ be the global minimizers of $E_m$ and $E_\infty$ defined in \eqref{eq: s.s. m} and \eqref{eq: s.s. infty}, respectively. Then, there exists a uniform positive constant, $C$, such that
    \begin{equation}
    \|\bar\varrho_m- \bar\varrho_\infty\|_{L^1(\R^d)} \leq \frac{C}{{m}}.
\end{equation}
\end{theorem}

\bigskip

\noindent Finally, we also provide examples of stationary solutions for which it can be shown that the rate of convergence in the $2$-Wasserstein distance is actually faster than $1/\sqrt m$. The condition under which we are able to find a finer estimate concerns the supports of the stationary solutions. In particular, we need the supports of $\bar\varrho_m$ to be included in $\rm{spt} (\bar\varrho_\infty)$ for all $m>1$. As we will show in Section~\ref{sec: global}, such solutions exist, provided that the potential $V$ is not ``too flat''.

\begin{theorem}[Faster $W_2$-rate for some stationary solutions]\label{prop: s.s. improved}
Let $V$ satisfy Assumption~\ref{assumptions potentials} with $\alpha>0$, and be such that the stationary solutions $\bar\varrho_m$ and $\bar\varrho_\infty$ defined in \eqref{eq: s.s. m}-\eqref{eq: s.s. infty} satisfy $\rm{spt}{\bar\varrho_m}\subset \rm{spt}{\bar\varrho_\infty}$, for $m\gg 1$. Then, for any parameter $q>d$, denoting $\kappa := \frac{d+q}{{q}(d+2)}$, we have
\begin{equation}\label{eq: improved}
    W_2(\bar\varrho_m,\bar\varrho_\infty) \leq {C}{m^{-\frac{1}{2(1-\kappa)}}}.
    \end{equation}  
\end{theorem}

\bigskip

\noindent From \eqref{eq: improved} it is clear that the rate is improved since $0<{\kappa}<1/2$, and hence the rate as $m\to\infty$ belongs to the interval $(1/2, 1)$.
The fastest rate in this range seems to be almost achieved for $d=2$. In this case, 
taking $p=2+\delta>d$, with $\delta\ll1$, we find 
\begin{equation*}
    \frac{1}{2(1-\kappa)}= 1 - \mathcal{O}(\delta).
\end{equation*}

\section{Preliminaries and preparatory results}\label{sec: prelim}

\subsection{Optimal transport toolbox}

Here we recall some basic definitions and tools from the theory of optimal transport that we will use later on. We refer the reader to \cite{OTAM,villaniOT} for more details. 
Given two probability measures $\mu, \nu \in \curlyP_2(\R^d)$, the 2-Wasserstein distance is defined as
\begin{equation}\label{wass}
    W_2(\mu, \nu) =  \inf \left\{\iint_{\R^d\times\R^d} |x-y|^2 \dx \gamma(x,y), \gamma \in \Pi(\mu,\nu)\right\}^{\frac 12},
\end{equation}
where $\Pi(\mu,\nu):=\left\{\gamma {\in\curlyP_2(\R^d\times\R^d)}:\  (\pi^x)_\#\gamma = \mu, (\pi^y)_\#\gamma = \nu \right\}$ and ${\pi^x, \pi^y:\R^d\times\R^d\to\R^d}$ are the canonical projections from $\R^d\times\R^d$ to $\R^d$. Moreover, under the assumption that $\mu\ll \curlyL^d$ (where $\curlyL^d$ denotes the Lebesgue measure {supported} on $\R^d$) it is known that the optimal transport plan $\bar\gamma$ is induced by a map $T:\R^d\to\R^d$ which is the gradient of a convex function, $T=\nabla u$. This means $\bar\gamma= (\mathrm{id}, T)_\#\mu$. The function $u:\R^d \to \R$ is given by $u(x)=|x|^2/2 -\varphi(x)$, where $\varphi$ is the so-called Kantorovich potential for the transport between $\mu$ and $\nu$, and is the solution of the dual problem to \eqref{wass}.
Let us recall also that the map $T$ is in fact the solution to
\begin{equation}
     \inf \left\{{\int_{\R^d}} |x-T(x)|^2 \dx \mu:\ \  T_\#\mu = \nu\right\}^{\frac 12},
\end{equation}
and it can be written as $T=\mathrm{id} - \nabla \varphi$. 
Equivalently, there exists $\psi$ such that $T^{-1}=\mathrm{id}-\nabla \psi$, and, equivalently, we can summarize these properties as follows
\begin{equation}
    (\mathrm{id} - \nabla \varphi)_\#\mu = \nu, \quad  (\mathrm{id} - \nabla \psi)_\#\nu = \mu.
\end{equation}
Let us now recall the following important property of the Wasserstein distance, which relies on the $c$-cyclical monotonicity condition {of the support of the optimal plan $\bar\gamma = (\mathrm{id}, T)_\#\mu$}, see \cite{BJR2007}. 

\begin{lemma}\label{lemma 1}
    Let $x\in \R^d$ be such that $|T(x)-x|>0$. There exists $\Omega\subset\R^d$ and $C>0$ such that $x\in \Omega$,  $|\Omega|\geq C |T(x)-x|^d$, and 
    \begin{equation}
        |T(y)-y|\geq \frac{|T(x)-x|}{2} \; \text{ for all } y\in \Omega.
    \end{equation}
\end{lemma}

\medskip
\noindent
We recall a well-known formula from optimal transport theory which is very useful in applications to evolution PDEs. This is a corollary of \cite[Theorem~8.4.7]{AGS2008}.
\begin{lemma}
Let $\varrho_i : (0,T)\to\mathcal{P}_2(\R^d)$ be absolutely continuous curves and $v_i: (0,+\infty) \times \R^d \to \R^d$ Borel vector fields satisfying
\begin{equation}
    \partialt{\varrho_i} +\nabla\cdot(\varrho_i v_i)=0,
\end{equation}
in the sense of distributions, $i=1,2$. Suppose that $\varrho_i(t) \ll \curlyL^d$ for every $t$.
Then, the following formula holds
\begin{equation}\label{recall_ dt W_2}
    \ddt \prt*{\frac 12 W_2^2(\varrho_1(t), \varrho_2(t))} = \int_{\R^d} \nabla \varphi^t(x) \cdot v_1(t,x) \dx{\varrho_1(t)} + \int_{\R^d} \nabla \psi^t(x) \cdot v_2(t,x) \dx{\varrho_2(t)},
\end{equation}
for $\mathcal{L}^1$-a.e. $t\in (0,+\infty)$, where $(\varphi^t,\psi^t)$ is any pair of Kantorovich potentials in the optimal transport problem of $\varrho_1(t)$ onto $\varrho_2(t)$, in particular $\varrho_2(t) = (\mathrm{id} - \nabla\varphi^t)_\#\varrho_1(t)$ and $\varrho_1(t) = (\mathrm{id} - \nabla\psi^t)_\#\varrho_2(t).$
\end{lemma}

\subsection{Preparatory results}

We show that even if we do not control the norm of $\varrho_m^{m-1}$ in $L^\infty$, we can still infer a uniform control of the $L^{2m-1}$-norm of $\varrho_m(t)$ by a locally integrable function of time. This result will be employed in the proof of the main result to treat the porous medium term.
 
\begin{proposition}\label{prop: grad power rho}
Let $\varrho_m$ be the solution of equation \eqref{main tot} with \eqref{eq: pressure power law} and initial data $\varrho_0$, satisfying the assumptions of Theorem~\ref{thm: main}.
Let the function $f:(0,+\infty) \to [0,+\infty)$ be defined as
 \begin{equation}\label{eq: f}
     f(t):=\int_{\R^d} \varrho_m (t)|\nabla p_m(t) + \nabla V+ \nabla W \star \varrho_m(t)|^2\dx x.
 \end{equation}
Then $f$ is uniformly bounded in $L^{1}(0,\infty)$ (independently of $m$). Moreover, for any $T>0$, there exists $C>0$ independent of $m$ such that for almost every $0<t<T$ it holds
 \begin{equation}\label{eq: bound}
      \int_{\R^d} \varrho_m(t) |\nabla p_m(t)|^2 \dx x \leq f(t) + C.
\end{equation}
Moreover, if $\alpha,\beta$ satisfy either \textit{(i)} or \textit{(ii)} in Assumption~\ref{assum: global}, the above bound holds globally in time, namely for almost every $t>0$.
\end{proposition}

\begin{proof}
Let us compute the dissipation of the energy associated to equation \eqref{main tot}, namely 
\begin{equation}
    \curlyE_m(\varrho_m) :=  \frac{1}{m-1}\int_{\R^d}\varrho_m^{{m}} \dx x +  \int_{\R^d} \varrho_m V \dx x + \frac 12 \int_{\R^d}\varrho_m W\star \varrho_m \dx x.
\end{equation}
We integrate in time to find
\begin{equation}
    \ddt \curlyE_m(\varrho_m) = - \int_{\R^d} \varrho_m |\nabla p_m +\nabla V + \nabla W \star \varrho_m|^2 \dx x,
\end{equation}
and integrate in time to find the following \textit{energy equality}
\begin{equation}\label{energy equality}
\begin{aligned}
       \int_0^t\!\! \int_{\R^d} \varrho_m |\nabla p_m + \nabla V + \nabla W\star \varrho_m|^2 \dx x \dx \tau &+  \frac{1}{m-1}\int_{\R^d} \varrho_m^m(t) \dx x\\[0.4em] 
      &+\int_{\R^d} \varrho_m(t) V \dx x +\frac 12\int_{\R^d} \varrho_m(t) W\star\varrho_m(t)\dx x \\[0.9em]
       =   \frac{1}{m-1}&\int_{\R^d} \varrho_0^m \dx x+\int_{\R^d}\varrho_0 V \dx x +\frac 12\int_{\R^d} \varrho_0 W\star\varrho_0\dx x.
\end{aligned}
\end{equation}
We remark that all these calculations are meaningful because of \cite[Theorem~11.2.8]{AGS2008}. Since we assumed that both $V$ and $W$ are nonnegative, and by assumption $\|\varrho_0\|_\infty\leq 1$, from the above equality, we deduce
\begin{equation}\label{eq: f in L1}
    \int^\infty_0 \! f(t) \dx t < \infty, \qquad \int_{\R^d} \varrho_m (t) V \dx x < \infty, \qquad \int_{\R^d} \varrho_m(t) W\star\varrho_m(t)\dx x <\infty,
\end{equation}
uniformly in $m$. 
Since $D^2V, D^2W$ are both uniformly bounded by assumption, we have
\begin{equation}\label{control with m2}
      \int_{\R^d} \varrho_m(t) (|\nabla V|^2 + |\nabla W\star \varrho_m|^2) \dx x\leq C \int_{\R^d} \varrho_m(t)  (|x|^2+1)\dx x.
\end{equation}
Now let us show that, under Assumption~\ref{assumptions potentials}, the equation preserves the control on the second moment locally in time. 
We compute
\begin{align*}
  \ddt \int_{\R^d} |x|^2 \varrho_m \dx x &= \int_{\R^d} |x|^2 \nabla\cdot(\varrho_m (\nabla p_m + \nabla V + \nabla W\star \varrho_m)) \dx x\\[0.3em]
  &= {2} \int_{\R^d} \varrho_m x \cdot (\nabla p_m + \nabla V + \nabla W\star \varrho_m)\dx x\\[0.3em]
  &\leq {\int_{\R^d} |x|^2 \varrho_m\dx x +  f(t)}.
\end{align*}
By Gr\"onwall's lemma we conclude $\varrho_m |x|^2 \in L^{\infty}(0,T; L^2(\R^d))$ uniformly in $m$, and so we have shown the first part of the thesis of this proposition.

\bigskip

To conclude the proof we show that under Assumption~\ref{assum: global} the result holds globally in time.
If \textit{(i)} holds, then $\alpha >0$ and from \eqref{control with m2} and \eqref{eq: f in L1}, we have
\begin{equation}
    \int_{\R^d} \varrho_m(t) (|\nabla V|^2 + |\nabla W\star \varrho_m|^2) \dx x\leq C \int_{\R^d} \varrho_m(t)  (|x|^2+1)\dx x \leq C \int_{\R^d} \varrho_m(t)  (V+1)\dx x\leq C.
\end{equation}
If \textit{(ii)} holds, namely $V=0$ and $\beta > 0$, we have
\begin{equation}
    \int_{\R^d} \varrho_m(t) |\nabla W\star \varrho_m(t)|^2 \dx x \leq C \int_{\R^d} \varrho_m(t) W\star \varrho_m(t)\dx x\leq C.
\end{equation}
Thanks to the above estimate, we have
\begin{align} 
      \int_{\R^d} \varrho_m(t) |\nabla p_m(t)|^2 \dx x&\leq \int_{\R^d} \varrho_m(t) |\nabla p_m(t)+\nabla V +\nabla W\star \varrho_m(t)|^2 \dx x + C\\[0.5em]
      &\leq f(t) + C,
\end{align}
and this concludes the proof.
\end{proof}

\noindent
Let us recall a useful property that will be employed in the proof of Proposition~\ref{prop: bounds of powers}.
\begin{lemma}\label{lemma}
    Let $u:\R^d\to \R$ be a nonnegative function such that $\int_{\R^d} u(x)\dx x=1.$ For all $1<r<p$ we have
    \begin{equation}
        \int_{\R^d} u^r(x) \dx x \leq \prt*{\int_{\R^d}u^p(x) \dx x}^{\frac{r-1}{p-1}}.
    \end{equation}
\end{lemma}
\begin{proof}
    The result is a simple consequence of the H\"older inequality applied to the function $u^{r-1}$ integrated against the measure $\dx u(x)$.
\end{proof}

 \noindent
Combining Lemma~\ref{lemma} and Proposition~\ref{prop: grad power rho} we are able to prove a technical result that will be crucial in the proof of the main result. In particular, it will help us control the term coming from the porous medium part in the estimate of the distance $W_2(\varrho_1,\varrho_2)$ between two solutions of equation~\eqref{eq: main} with different exponents $1<m_1< m_2$. 

\begin{proposition}\label{prop: bounds of powers}
Let $m_{1}>1$ and set $m_2:=2m_1-1$. Let $\varrho_1, \varrho_2$ be solutions of equation \eqref{main tot} with pressure law \eqref{eq: pressure power law} with exponents $m_1,m_2$, respectively. There exist uniform positive constants $C$ and $C_1$ such that
\begin{equation}
    \int_{\R^d} \varrho_1^{m_2}(t) \dx x + \int_{\R^d} \varrho_2^{m_1}(t) \dx x \leq C_1(C + f_1(t)+f_2(t)),
\end{equation}
where $f_i$, associated to $\varrho_i$ for $i=1,2$, is given by \eqref{eq: f}.
\end{proposition}

\begin{proof}
We may use Lemma~\ref{lemma} with $r=2m_1-1$ and $p=2^*(m_1 - 1/2)$ to find
\begin{equation}\label{eq: 2m-1}
\begin{aligned}
    \int_{\R^d} \varrho_1^{2m_1-1}(t) \dx x &\leq \prt*{\int_{\R^d} \varrho_1^{p}(t) \dx x}^{\frac{r-1}{p-1}}\\[0.6em] 
    &\leq C_S \prt*{\int_{\R^d} \left|\nabla \varrho_1^{m_1-1/2}(t)\right|^2 \dx x}^{\frac{2^*(r-1)}{2(p-1)}},
\end{aligned}
\end{equation}
where in the last line we used the Sobolev inequality.
Since
\begin{equation}
    \left|\frac{2m_{1}}{2m_{1}-1}\nabla \varrho_1^{m_1-1/2}\right|^2 = \varrho_1\left|\nabla p_1\right|^2,
\end{equation}
we may use the control given by equation \eqref{eq: bound} in \eqref{eq: 2m-1} to infer
\begin{align*}
    \int_{\R^d} \varrho_1^{2m_1-1}(t) \dx x \leq C_S\prt*{C+ f_{1}(t)}^{\frac{2^*(r-1)}{2(p-1)}}.
\end{align*}
From now on $C$ denotes a positive constant whose value may change from line to line. By definition of $r$ and $p$ we have
\begin{equation}
    \frac{2^*(r-1)}{2(p-1)}= \frac{2^* (2m_1 -2)}{2^* (2m_1 - 1)  -2} \leq 1,
\end{equation}
thus, taking $C$ large enough we have
\begin{align*}
    \int_{\R^d} \varrho_1^{2 m_1-1}(t) \dx x \leq C_S\prt*{C+ f_{1}(t)}.
\end{align*}
Let us notice that by arguing in the same way we may find
\begin{align*}
    \int_{\R^d} \varrho_2^{2 m_2-1}(t) \dx x \leq C_S\prt*{C+ f_{2}(t)}.
\end{align*} 
Using again Lemma~\ref{lemma} with $r=(m_2+1)/2$ and $p=2m_2-1$, we obtain
\begin{align}
   \int_{\R^d} \varrho_2^{m_1}(t) \dx x &= \int_{\R^d} \varrho_2^{(m_2+1)/2}(t) \dx x \\[0.6em]
    &\leq \prt*{\int_{\R^d} \varrho_2^{2m_2-1}(t) \dx x }^{q}\\[0.6em]
    &\leq \prt*{C_S(f_{2}(t)+C)}^q\\[0.6em]
    &\leq C_S^q(f_{2}(t)+C),
\end{align}
where $q=1/4$, and we again assumed that the constant $C$ is large enough.
\end{proof}

\section{Rate of convergence in the 2-Wasserstein distance}\label{sec: global}

We now prove the main result of the paper, namely that the rate of convergence in the 2-Wasserstein distance for the incompressible limit is at least polynomial with exponent $1/2$. 

\begin{proof}[Proof of Theorem~\ref{thm: main}]

We begin by proving the result on the singular pressure law \eqref{singular pressure}. The main difference consists in the fact that for this law we can directly compute the distance between $\varrho_\varepsilon$ (solution of \eqref{main tot}) and $\varrho_\infty$ solution of the limit equation
\begin{equation}
\begin{aligned}
    &\partialt{\varrho_\infty} = \Delta p_\infty + \nabla \cdot(\varrho_\infty (\nabla V + \nabla W \star \varrho_\infty)),\\[0.6em]
    &p_\infty(1-\varrho_\infty)=0.
\end{aligned}
\end{equation}
This is possible due to the fact that $\varrho_\varepsilon<1 \text{ almost everywhere in } {(0,\infty)\times}\R^d,$ for all $\varepsilon>0$, while it is not true for the porous medium law \eqref{eq: pressure power law}. We will underline in the proof where this fact is used.

\smallskip

\noindent
{\textbf{{Singular pressure law.}}}
By the formula recalled in the previous section, see \eqref{recall_ dt W_2}, we have
\begin{equation}\label{dt}
    \ddt \prt*{\frac 12 W_2^2(\varrho_\varepsilon(t), \varrho_\infty(t))} = \int_{\R^d} \nabla \varphi^t \cdot v^t_\varepsilon \dx{\varrho_\varepsilon(t)} + \int_{\R^d} \nabla \psi^t \cdot v^t_\infty \dx{\varrho_\infty(t)},
\end{equation}
where $(\varphi^t, \psi^t)$ is a pair of Kantorovich potentials in the transport between $\varrho_\varepsilon(t)$ and $\varrho_\infty(t)$ with quadratic cost and $v^t_\varepsilon = -\nabla p_\varepsilon(t) - \nabla V - \nabla W \star \varrho_\varepsilon(t),$ and $v^t_\infty = -\nabla p_\infty(t) - \nabla V - \nabla W \star \varrho_\infty(t).$ Let us recall that this is true since $\nabla p_\infty(t) = \varrho_\infty(t)\nabla p_\infty(t)$.
The optimal transport map in the transport between $\varrho_\varepsilon(t)$ and $\varrho_\infty(t)$ is given by $T= \mathrm{id}  - \nabla \varphi^t$, with $\mathrm{id}  - \nabla \psi^t$ the inverse transport map.
Let us note that the nonlinear diffusion part of the equation can be written as the Laplacian of a positive function $H:[0,1)\to\R$ defined as
\begin{equation}\label{eq: H}
    H(\varrho) = \frac{\varrho}{1-\varrho} + \ln (1-\varrho),
\end{equation}
hence \eqref{main tot} can be written as
\begin{equation}\label{main tot es}
    \partialt{\varrho_\varepsilon} = \varepsilon\Delta H(\varrho_\varepsilon)+ \nabla \cdot(\varrho_\varepsilon( \nabla V + \nabla W \star \varrho_\varepsilon)).
\end{equation}
After integration by parts, from \eqref{dt} we obtain
\begin{equation}\label{dt W2}
\begin{aligned}
    \ddt \prt*{\frac 12 W_2^2(\varrho_\varepsilon(t), \varrho_\infty(t))}  &\leq  \int_{\R^d} \prt*{\varepsilon \Delta \varphi^t H(\varrho_\varepsilon(t)) +\Delta \psi^t p_\infty(t) }\dx x\\[0.6em]
    & - \int_{\R^d} (\nabla \varphi^t \cdot \nabla V \varrho_\varepsilon(t)  + \nabla \psi^t \cdot \nabla V \varrho_\infty(t)) \dx x\\[0.6em]
    &  - \int_{\R^d} (\nabla \varphi^t \cdot \varrho_\varepsilon(t)\nabla W\star \varrho_\varepsilon(t)    +\nabla \psi^t \cdot \varrho_\infty(t)\nabla W\star \varrho_\infty(t)) \dx x.
\end{aligned}
\end{equation}
Let us recall that the Laplacian of the Kantorovich potentials is a measure on $\R^d$. However, we know that the singular part of these measures is negative, while $H(\cdot)\geq 0$ and $p_\infty\geq 0$. Therefore, the first inequality in \eqref{dt W2} results from dropping the singular part of the measures (for the sake of readability we still denote $\Delta \varphi^t, \Delta \psi^t$ the absolutely continuous part of the measures). The integrability of the term $\Delta \varphi^t H(\varrho_\varepsilon(t))$ will follow from the computations below, see \eqref{singular terms}.

\smallskip

\textit{\underline{\smash{Drift part.}}} Let us now treat the term involving the potential $V$.
Since $\nabla \psi (T(x))=- \nabla \varphi(x)$, using the fact that $\varrho_\infty(t) = T_\#\varrho_\varepsilon(t)$, we have
\begin{align*}
- \int_{\R^d} \nabla \varphi^t(x)& \cdot \nabla V(x) \varrho_\varepsilon(t,x)\dx x -\int_{\R^d} \nabla \psi^t(y) \cdot \nabla V(y) \varrho_\infty(t,y) \dx y\\[0.6em]
&=- \int_{\R^d} \nabla \varphi^t(x) \cdot \nabla V(x) \varrho_\varepsilon(t,x)\dx x - \int_{\R^d} \nabla \psi^t(T(x)) \cdot \nabla V(T(x)) \varrho_\varepsilon(t,x) \dx x\\[0.6em]
 &=- \int_{\R^d} \nabla \varphi^t(x) \cdot (\nabla V(x)-  \nabla V(T(x)) \varrho_\varepsilon(t,x) \dx x\\[0.6em]
 &= \int_{\R^d} (T(x)-x)  \cdot (\nabla V(x)-  \nabla V(T(x)) \varrho_\varepsilon(t,x) \dx x\\[0.6em]
&\leq - \alpha \int_{\R^d} |T(x)-x|^2 \dx \varrho_\varepsilon(t,x)\\[0.6em]
    &= -\alpha W_2^2(\varrho_\varepsilon(t),\varrho_\infty(t)).
\end{align*}
where we used the fact that
\begin{equation}
(T(x)- x) \cdot (\nabla V(T(x))- \nabla V(x)) \geq \alpha |T(x)- x|^2,
\end{equation}
since, by assumption, $D^2 V\ge \alpha I$.

\medskip
\noindent
\textit{\underline{\smash{Interaction potential part.}}} For the nonlocal term we argue as follows
\begin{equation}\label{est: nonlocal}
    \begin{aligned}
        &- \int_{\R^d} \nabla \varphi^t (x) \cdot \nabla W\star \varrho_\varepsilon (t,x) \varrho_\varepsilon(t,x) \dx x  - \int_{\R^d} \nabla \psi^t (y) \cdot \nabla W\star \varrho_\infty (t,y) \varrho_\infty(t,y) \dx y\\[0.5em]
        =&- \int_{\R^d} \nabla \varphi^t (x) \cdot \nabla W\star \varrho_\varepsilon (t,x) \varrho_\varepsilon(t,x) \dx x  + \int_{\R^d} \nabla \varphi^t (x) \cdot \nabla W\star \varrho_\infty (T(x)) \varrho_\varepsilon(t,x) \dx x\\[0.5em]  
     =&\int_{\R^d} \nabla \varphi^t (x) \cdot (\nabla W\star \varrho_\infty (t,T(x)) - \nabla W(x) \star\varrho_\varepsilon(t,x)) \varrho_\varepsilon(t,x) \dx x.
    \end{aligned}
\end{equation}
We now compute
\begin{equation}
    \begin{aligned}
    (\nabla W\star \varrho_\infty (t,T(x)) - \nabla W(x) \star\varrho_\varepsilon(t,x)) =& \int_{\R^d} \nabla W(T(x)-z) \varrho_\infty(t,z) \dx z\\[0.5em]
    & - \int_{\R^d} \nabla W(x-y) \varrho_\varepsilon(t,y) \dx y\\[0.5em]
    =& \int_{\R^d} \nabla W(T(x)-T(y)) - \nabla W(x-y) )\varrho_\varepsilon(t,y) \dx y.
    \end{aligned}
\end{equation}
Coming back to \eqref{est: nonlocal}, we find
\begin{equation}
\begin{aligned}
    &\iint_{\R^d\times\R^d} (x- T(x))\cdot (\nabla W(T(x)-T(y)) - \nabla W(x-y)) \varrho_\varepsilon(t,y) \varrho_\varepsilon(t,x) \dx y \dx x\\[0.5em]
    =&\iint_{\R^d\times\R^d} (x-y - (T(x)-T(y)))\cdot (\nabla W(T(x)-T(y)) - \nabla W(x-y)) \varrho_\varepsilon(t,y) \varrho_\varepsilon(t,x) \dx y \dx x\\[0.5em] 
     &\qquad + \iint_{\R^d\times\R^d} (y -T(y))\cdot (\nabla W(T(x)-T(y)) - \nabla W(x-y)) \varrho_\varepsilon(t,y) \varrho_\varepsilon(t,x) \dx y \dx x.
     \end{aligned}
\end{equation}
The last integral in the right-hand side is exactly equal to the opposite of the left-hand side since $W(x)=W(-x)$ (and hence $\nabla W(x)=-\nabla W(-x)$). Therefore, we obtain
\begin{equation}
    \begin{aligned}
    &\iint_{\R^d\times\R^d} (x- T(x))\cdot (\nabla W(T(x)-T(y)) - \nabla W(x-y)) \varrho_\varepsilon(t,y) \varrho_\varepsilon(t,x) \dx y \dx x\\[0.5em]
    =&\frac 12 \iint_{\R^d\times\R^d} (x-y - (T(x)-T(y)))\cdot (\nabla W(T(x)-T(y)) - \nabla W(x-y)) \varrho_\varepsilon(t,y) \varrho_\varepsilon(t,x) \dx y \dx x\\[0.5em]
   \leq &-\frac \beta 2 \iint_{\R^d\times\R^d}|x-y - (T(x)-T(y))|^2\varrho_\varepsilon(t,y) \varrho_\varepsilon(t,x) \dx y \dx x,
     \end{aligned}
\end{equation}
thanks to the assumption $D^2W\geq \beta I_d$.
It now remains to compute the last integral
\begin{equation}
    \begin{split}
   -\frac \beta 2 \iint_{\R^d\times\R^d}&|x-y - (T(x)-T(y))|^2\varrho_\varepsilon(t,y) \varrho_\varepsilon(t,x) \dx y \dx x\\[0.8em]
   &=- \frac\beta 2 \int_{\R^d} |x-T(x)|^2 \varrho_\varepsilon(t,x)\dx x -\frac \beta 2 \int_{\R^d} |y-T(y)|^2\varrho_{\varepsilon}(t,y) \dx y\\[0.5em]
   &\qquad\quad + \beta \iint_{\R^d\times\R^d}(x-T(x))\cdot (y-T(y))\varrho_\varepsilon(t,y) \varrho_\varepsilon(t,x) \dx y \dx x\\[0.8em]
   &=-\beta W_2^2(\varrho_\varepsilon(t), \varrho_\infty(t)) + \beta \int_{\R^d} (x - T(x))\varrho_\varepsilon(t,x) \dx x\int_{\R^d} (y - T(y))\varrho_{\varepsilon}(t,y) \dx y\\[0.8em]
   &=-\beta W_2^2(\varrho_\varepsilon(t), \varrho_\infty(t)) + \beta|\mathrm{bar}(\varrho_\varepsilon(t))- \mathrm{bar}(\varrho_\infty(t))|^2,
    \end{split}
\end{equation}
where $\mathrm{bar}(\varrho)\in\R^d$ denotes the center of mass of $\varrho$. Therefore, we conclude
\begin{align*}
   - \int_{\R^d} \nabla \varphi^t (x) \cdot \nabla W\star \varrho_\varepsilon (t,x) \varrho_\varepsilon(t,x) \dx x&  - \int_{\R^d} \nabla \psi^t (y) \cdot \nabla W\star \varrho_\infty (t,y) \varrho_\infty(t,y) \dx y\\
   & \leq -\beta W_2^2(\varrho_\varepsilon,\varrho_\infty) + \beta |\mathrm{bar}(\varrho_\varepsilon(t))- \mathrm{bar}(\varrho_\infty(t))|^2.
\end{align*}
Coming back to \eqref{dt W2}, we find
\begin{equation}
\begin{aligned}\label{final_singular}
    \ddt \prt*{\frac 12 W_2^2(\varrho_\varepsilon(t), \varrho_\infty(t))} &\leq \int_{\R^d} \prt*{\varepsilon\Delta \varphi^t H(\varrho_\varepsilon(t)) +\Delta \psi^t p_\infty(t) }\dx x \\[0.5em]
    &\qquad\underbrace{- (\alpha+\beta) W_2^2(\varrho_\varepsilon(t), \varrho_\infty(t)) + \beta |\mathrm{bar}(\varrho_\varepsilon(t))- \mathrm{bar}(\varrho_\infty(t))|^2}_{I}.
\end{aligned}
\end{equation}
Now we treat the term coming from the degenerate diffusion.

\smallskip
\noindent
\textit{\underline{\smash{Nonlinear diffusion part.}}}
By the Monge--Amp\`ere equation, we have
\begin{equation}
    \det (I - D^2 \varphi^t(x)) = \frac{\varrho_\varepsilon(t,x)}{\varrho_\infty(t,T (x))}, \quad \varrho_{\varepsilon}(t,\cdot)-\mathrm{a.e.}, \quad   \det (I - D^2 \psi^t(y)) = \frac{\varrho_\infty(t,y)}{\varrho_\varepsilon(t,T^{-1}(y))} \quad \varrho_{\infty}(t,\cdot)-\mathrm{a.e.}
\end{equation}
We use the inequality between the arithmetic and geometric means, namely $\mathrm{tr} (A) \geq d \det (A)^{1/d}$ for any symmetric positive semi-definite matrix $A$, to find
\begin{equation}
     \Delta \varphi^t(x)\leq d\prt*{1-\prt*{\frac{\varrho_\varepsilon(t,x)}{\varrho_\infty(t,T(x))}}^{1/d}}, \quad  \Delta \psi^t(y)\leq d\prt*{1-\prt*{\frac{\varrho_\infty(t,y)}{\varrho_\varepsilon(t,T^{-1}(y))}}^{1/d}},
\end{equation}
where the two inequalities take place pointwise a.e. on the supports of the corresponding measures.
First of all, we notice that if $p_\infty(t)>0$ then $\Delta \psi^t\le 0$ since $\varrho_\infty(t)=1> \varrho_\varepsilon(t)$ (we see that we use here the inequality $ \varrho_\varepsilon<1$). Thus we have
\begin{equation*}
\begin{split}
    \int_{\R^d} \prt*{\varepsilon \Delta \varphi^t H(\varrho_\varepsilon(t)) +\Delta \psi^t p_\infty(t) }\dx x 
    &\leq d \varepsilon \int_{\R^d} \prt*{1-\prt*{\frac{\varrho_\varepsilon(t,x)}{\varrho_\infty(t,T(x))}}^{1/d}} H(\varrho_\varepsilon(t,x)) \dx x\\[0.5em]
    &= {d \varepsilon \int_{\{\varrho_\varepsilon(t) \geq \varrho_\infty(t)\circ T\}} \prt*{1-\prt*{\frac{\varrho_\varepsilon(t,x)}{\varrho_\infty(t,T(x))}}^{1/d}} H(\varrho_\varepsilon(t,x)) \dx x}\\[0.5em]
    &\quad
    + {d \varepsilon \int_{\{\varrho_\varepsilon(t) < \varrho_\infty(t)\circ T\}} \prt*{1-\prt*{\frac{\varrho_\varepsilon(t,x)}{\varrho_\infty(t,T(x))}}^{1/d}} H(\varrho_\varepsilon(t,x)) \dx x}.
\end{split}
\end{equation*}
We notice that $H$ is always nonnegative on $[0,1)$, and therefore, the first contribution on the right-hand side of the previous equality can be neglected. Therefore, we obtain
\begin{equation}
\begin{split}
    \int_{\R^d} \prt*{\varepsilon \Delta \varphi^t H(\varrho_\varepsilon(t)) +\Delta \psi^t p_\infty(t) }\dx x 
    &\le {d \varepsilon \int_{\{\varrho_\varepsilon(t) < \varrho_\infty(t)\circ T\}} \prt*{1-\prt*{\frac{\varrho_\varepsilon(t,x)}{\varrho_\infty(t,T(x))}}^{1/d}} H(\varrho_\varepsilon(t,x)) \dx x}\\[0.5em]
    &\leq d \varepsilon \int_{{\{\varrho_\varepsilon(t) < \varrho_\infty(t)\circ T\}}} \prt*{1-\prt*{\frac{\varrho_\varepsilon(t,x)}{\varrho_\infty(t,T(x))}}^{1/d}} \frac{\varrho_\varepsilon(t,x)}{1-\varrho_\varepsilon(t,x)} \dx x.
\end{split}
\end{equation}
So, we have
    \begin{equation}
\label{singular terms}
\begin{split}
        d \varepsilon \int_{{\{\varrho_\varepsilon(t) < \varrho_\infty(t)\circ T\}}} & \prt*{1-\prt*{\frac{\varrho_\varepsilon(t,x)}{\varrho_\infty(t,T(x))}}^{1/d}} \frac{\varrho_\varepsilon(t,x)}{1-\varrho_\varepsilon(t,x)} \dx x\\[0.5em]
        &\le d \varepsilon \int_{\R^d} \varrho_\varepsilon(t,x) \dx x\\[0.5em]
        &= d \varepsilon,
    \end{split}
    \end{equation}
where the last inequality follows by observing that, since $a<b\leq 1$, we have
$$\left(1-\frac ab\right)\frac{1}{1-a^d}\leq 1.$$
Let us come back to \eqref{final_singular}.
We notice that we can bound the term $I$ on the right-hand side in the following way
\begin{equation}
    I\leq -\gamma W_2^2(\varrho_\varepsilon(t),\varrho_\infty(t)), \quad 
\text{ with } \quad
    \gamma := \begin{cases}
    \alpha+ \beta,    &\text{ if } \beta\leq 0,\\
  \alpha, &\text{ if }   \beta>0,\\
  \beta, &\text{ if } V=0.
    \end{cases}
\end{equation}
This follows by the fact that $|\mathrm{bar}(\varrho_\varepsilon(t))- \mathrm{bar}(\varrho_\infty(t))|^2\leq W^2_2(\varrho_\varepsilon(t),\varrho_\infty(t))$, and that if $V=0$, the center of mass is preserved, see Remark~\ref{rmk: center mass}, hence $|\mathrm{bar}(\varrho_\varepsilon(t))- \mathrm{bar}(\varrho_\infty(t))|^2=0$. 
Indeed, to see the inequality involving the barycenters, we could argue as follows. For any two measures $\mu,\nu \in \mathcal{P}_{2}(\R^{d})$, consider $\pi_{0}\in\Pi(\mu,\nu)$ any optimal transport plan for the distance $W_{2}$. Then, we have
\begin{align*}
\left| \mathrm{bar}(\mu) - \mathrm{bar}(\nu)\right| & = \left| \int_{\R^{d}}x\dx\mu(x) - \int_{\R^{d}} y \dx \nu(y)\right| = \left| \iint_{\R^{d}\times\R^{d}}x\dx\pi_{0}(x,y) - \iint_{\R^{d}\times\R^{d}} y \dx \pi_{0}(x,y)\right|\\
&\le \iint_{\R^{d}\times\R^{d}}|x-y|\dx\pi_{0}(x,y) \le W_{2}(\mu,\nu),
\end{align*}
where in the last inequality we have used the Cauchy--Schwarz inequality.

Therefore, we obtain
\begin{equation}
    \ddt \left(\frac 12 W_2^2(\varrho_\varepsilon(t),\varrho_\infty(t)\right)\leq -\gamma W_2^2(\varrho_\varepsilon(t),\varrho_\infty(t)) + d \varepsilon.
\end{equation}

By Gr\"onwall lemma we conclude the proof since $\gamma>0$ under the assumption imposed on $\alpha$ and $\beta$ in the statement of Theorem~\ref{thm: main}. 

Let us now prove the result for the porous medium equation, namely \eqref{main tot} with~\eqref{eq: pressure power law}.

\smallskip
\noindent
\textbf{Porous medium equation.}
Let $m_1, m_2>1$ and $\varrho_1, \varrho_2$ be solutions to equation \eqref{main tot} with exponent $m_1$ and $m_2$, respectively. We have
\begin{equation}
    \ddt \prt*{\frac 12 W_2^2(\varrho_1(t), \varrho_2(t))} = \int_{\R^d} \nabla \varphi^t \cdot v_1^t \dx{\varrho_1(t)} + \int_{\R^d} \nabla \psi^t \cdot v^t_2 \dx{\varrho_2(t)},
\end{equation}
where $(\varphi^t, \psi^t)$ is the pair of Kantorovich potentials in the transport between $\varrho_1(t)$ and $\varrho_2(t)$ with quadratic cost and $v_i^t = -\nabla p_i(t) - \nabla V - \nabla W \star \varrho_i(t),$ for $i=1,2$.
Arguing in the same way as before, we obtain
\begin{equation}\label{dt W2 pme}
\begin{aligned}
    \ddt \prt*{\frac 12 W_2^2(\varrho_1(t), \varrho_2(t))}  &= \int_{\R^d} \prt*{\Delta \varphi^t \varrho_1^{m_1}(t) +\Delta \psi^t \varrho_2^{m_2}(t) }\dx x\\[0.6em] 
    &\quad - \int_{\R^d} (\nabla \varphi^t \cdot \nabla V \varrho_1(t)  + \nabla \psi^t \cdot \nabla V \varrho_2(t)) \dx x\\[0.6em]
    & \quad - \int_{\R^d} (\nabla \varphi^t \cdot \varrho_1(t)\nabla W\star \varrho_1(t)    +\nabla \psi^t \cdot \varrho_2(t)\nabla W\star \varrho_2(t) ) \dx x\\[0.6em]
    &\leq \int_{\R^d} \prt*{\Delta \varphi \varrho_1^{m_1}(t) +\Delta \psi \varrho_2^{m_2}(t) }\dx x -\gamma W_2^2(\varrho_1(t),\varrho_2(t)),
\end{aligned}
\end{equation} 
where $\gamma$ is defined as before. 
Now we treat the porous medium part of the equation.
Let us recall the Monge--Amp\`ere equations
\begin{equation}
    \det (I - D^2 \varphi^t(x)) = \frac{\varrho_1(t,x)}{\varrho_2(t,T (x))}, \quad \varrho_1(t)-a.e. \qquad  \det (I - D^2 \psi^t(y)) = \frac{\varrho_2(t,y)}{\varrho_1(t,T^{-1}(y))} \quad \varrho_2(t)-a.e.,
\end{equation}
and that we have
\begin{equation}
     \Delta \varphi^t(x)\leq d\prt*{1-\prt*{\frac{\varrho_1(t,x)}{\varrho_2(t,T(x))}}^{1/d}}, \quad  \Delta \psi^t(y)\leq d\prt*{1-\prt*{\frac{\varrho_2(t,y)}{\varrho_1(t,T^{-1}(y))}}^{1/d}},
\end{equation}
pointwise a.e. on the supports of the corresponding measures. Note that here, the fact that we do not know that solutions with $m>1$ are smaller than $1$ makes the proof more complicated compared to the singular pressure case.
Thus, we can estimate the integral as follows
\begin{align*}
   \int_{\R^d} \prt*{\Delta \varphi^t \varrho_1^{m_1}(t) +\Delta \psi^t \varrho_2^{m_2}(t) }\dx x &\leq  d \int_{\R^d} \prt*{1-\prt*{\frac{\varrho_1(t,x)}{\varrho_2(t,T(x))}}^{1/d}} \varrho_1^{m_1}(t,x)\dx x \\[0.6em]
    &\qquad+ d \int_{\R^d} \prt*{1-\prt*{\frac{\varrho_2(t,y)}{\varrho_1(t,T^{-1}(y))}}^{1/d}} \varrho_2^{m_2-1}(t,y) \varrho_2(t,y)\dx y\\[0.6em]
    &= d \int_{\R^d} \varrho_1(t,x) \prt*{1-\prt*{\frac{\varrho_1(t,x)}{\varrho_2(t,T(x))}}^{1/d}}  \varrho_1^{m_1-1}(t,x) \dx x \\[0.6em]
    &\qquad\quad + d\int_{\R^d}\varrho_1(t,x)\prt*{1-\prt*{\frac{\varrho_2(t,T(x))}{\varrho_1(t,x)}}^{1/d}} \varrho_2^{m_2-1}(t,T(x))\dx x.
\end{align*}
Let us denote $a=a(t,x):= \varrho_1(t,x)^{1/d}$ and $b=b(t,x):=\varrho_2(t,T(x))^{1/d}$. The last two integrals can be rewritten as
\begin{equation}
    d \int_{\R^d} \varrho_1(t,x) \Big[\underbrace{\prt*{1-\frac ab } a^{d(m_1-1)}}_{=:\curlyI_1} + \underbrace{\prt*{1-\frac ba} b^{d(m_2-1)}}_{=:\curlyI_2}\Big]\dx x = d \int_{\R^d} \varrho_1(t,x) (\curlyI_1 + \curlyI_2) \dx x.
\end{equation}
A direct computation shows that if $a\ge b$, then $\curlyI_1\le 0$ and
\begin{align}
    \max_b \curlyI_2 &= \frac{1}{d(m_2-1)+1} a^{d(m_2-1)} \left(\frac{d(m_2-1)}{d(m_2-1)+1}\right)^{d(m_2-1)} \\[0.6em]
    &\leq C \frac{1}{d(m_2-1)+1}a^{d(m_2-1)},
\end{align}
while if $b\ge a$, then $\curlyI_2\le 0$ and
\begin{align}
    \max_b \curlyI_1 &= \frac{1}{d(m_1-1)+1} b^{d(m_1-1)} \left(\frac{d(m_1-1)}{d(m_1-1)+1}\right)^{d(m_1-1)} \\[0.6em] 
    &\leq C \frac{1}{d(m_1-1)+1}b^{d(m_1-1)}.
\end{align}
Thus
\begin{align}\label{intermediate}
    \ddt \prt*{\frac 12 W_2^2(\varrho_1(t), \varrho_2(t))} &\leq \!- \gamma W_2^2(\varrho_1(t), \varrho_2(t))\\
    &+ \frac{C}{d\min\{m_1,m_2\}+1} \int_{\R^d}\! \varrho_1(t,x)\left\{\varrho_1^{m_2-1}(t,x)+\varrho_2^{m_1-1}(t,T(x))\right\}\dx x,
\end{align}
Therefore, the argument boils down to estimating the two integrals
\begin{equation}
    \int_{\R^d} \varrho_1^{m_2}(t,x) \dx x, \; \text{ and } \; \int_{\R^d} \varrho_1(t,x) \varrho_2^{m_1-1}(t,T(x)) \dx x.
\end{equation}
Notice that the second integral is 
\begin{align}
    \int_{\R^d} \varrho_1(t,x)\varrho_2^{m_1-1}(t,T(x)) \dx x &= \int_{\R^d} \varrho_2^{m_1}(t,x) \dx x,
    \end{align}
therefore, assuming $m_2=2m_1-1$, the bound is given by Proposition~\ref{prop: bounds of powers}.  We see here the reason for using two finite exponents $m_1,m_2$: we need to guarantee the summability of the solution of one equations raised to a different power than $m$ itself, and the estimate of Proposition~\ref{prop: bounds of powers} requires to use a power which is not too large compared to $m$, making it impossible to directly compare a solution to $\varrho_\infty$.
Therefore we have proven the following inequality
\begin{align}\label{intermediate 2}
    \ddt \prt*{\frac 12 W_2^2(\varrho_1(t), \varrho_2(t))} &\leq - \gamma W_2^2(\varrho_1(t), \varrho_2(t))+ \frac{C_1(f(t) + C)}{m_1},
\end{align}
for some $C_1, C>0$, and where $f:=f_1+f_2$. Without loss of generality we can assume $C_1=1$. Now we conclude by using Gr\"onwall's lemma. For the sake of clarity let us write the full argument.
We denote $$X(t):= W_2^2(\varrho_1(t),\varrho_2(t)), \quad g(t):= {2}(f(t)+C).$$
Then we have
\begin{equation}
    X'(t)\leq -2\gamma X(t) + \frac{g(t)}{m_1}.
\end{equation}
By computing  $(e^{2\gamma t}X(t))'$ and integrating in time between $0$ and $t$ we have
\begin{equation}
    X(t) \leq \frac{e^{-2\gamma t}}{m_1}\int_0^t e^{2\gamma s} g(s)\dx s,
\end{equation}
since we chose for simplicity $\varrho_0$ to be independent of the exponent of the porous medium term, and hence $X(0)=0$. For $\gamma>0$ (which is true under the assumptions on $\alpha$ and $\beta$ of Theorem~\ref{thm: main}) we get
\begin{equation}
    X(t) \leq  \frac{2}{m_1}\left(\int_0^t f(s)\dx s + \frac{C}{\gamma} \right)\leq \frac{C}{m_1},
\end{equation}
where $C$ is independent of the final time $T$. Otherwise, for $\gamma<0$ we have
\begin{equation}
    X(t) \leq  \frac{e^{-2\gamma t}}{m_1}\int_0^t g(s)\dx s\leq \frac{C(T)}{m_1}.
\end{equation}
Let $\varrho_m$ be the solution of equation \eqref{eq: main}. Thanks to the above estimates, we can conclude by noticing that the triangle inequality yields that
\begin{equation}
    W_2^2(\varrho_m(t), \varrho_\infty(t)) \le 2   \sum_{k=0}^{\infty} W_2^2(\varrho_{2^km}(t), \varrho_{2^{k+1} m-1}(t)) \leq \sum_{k=0}^{\infty} \frac{C}{2^{k} m} \leq \frac{C}{m},
\end{equation}
where the bound is local or global in time for $\gamma<0$ and $\gamma>0$, respectively.
\end{proof}

\section{Rate of convergence for stationary states}\label{sec:5}

\subsection{Rate in the $L^1$ norm }

Let us now discuss some results on the convergence of the stationary solutions of equation \eqref{eq: main} for convex potentials, namely $\alpha>0$ in Assumption~\ref{assumptions potentials}. 
In this section, we are interested in bounding the $W_2$ distance between two stationary states via a general estimate which allows us to bound this distance in terms of the $L^1$ distance. For this reason, we named this section ``Rate in the $L^1$ norm'', because what we essentially do is to obtain a bound in $L^1$.

\begin{remark}    
Let $V$ satisfy Assumption~\ref{assumptions potentials} with $\alpha>0$ (therefore the assumption $\inf_x V(x)=V(0)=0$ is no longer restrictive). This implies there exist positive constants $k_1> k_2$ such that
\begin{itemize}
    \item $ |\nabla V|^2 \leq  A^2 |x|^2 \leq k_1 V,$
    \item $|\nabla V|^2 \geq \alpha^2 |x|^2 \geq k_2 V.$
\end{itemize} 
We underline that from here until the end of the manuscript the interaction potential is taken to be constant zero, i.e. $W\equiv 0$.
\end{remark}

Let us recall that the stationary states for $m>1$ and $m=\infty$ have the following explicit forms
\begin{equation}\label{def:stat}
    \bar\varrho_m(x) = \prt*{\frac{m-1}{m}(C_m - V(x))_+}^{\frac{1}{m-1}}, \quad \bar\varrho_\infty(x)=\mathds{1}_{\{C_\infty > V(x)\}}.
\end{equation} 
Indeed, these formulas are simple consequences of the fact that $\bar\varrho_m$ and $\bar\varrho_{\infty}$ are the global minimizers of the free energies $E_{m}$ and $E_{\infty}$, respectively, and we can derive them from the the first order optimality conditions. Here $C_{m}$ and $C_{\infty}$ are explicit constants which make both $\bar\varrho_m$ and $\bar\varrho_\infty$ probability measures.
\begin{lemma}\label{lem:sec5}
Consider the stationary states defined in \eqref{def:stat}. Then there exists a constant $C>0$ independent of $m$ such that
$$
\|\partial_m \bar\varrho_m\|_{L^1} \leq \frac{C}{(m-1)^2}\ \ {\rm{and}}\ \ \left| \partial_m W_2^2(\bar\varrho_m, \bar\varrho_\infty)\right| \leq \frac{C}{(m-1)^2}.
$$ 
\end{lemma}

\begin{proof}

We start by computing $\|\partial_m \bar\varrho_m\|_{L^1}$. In order to ease this computation, we begin with the following observation.
Since the mass is conserved, we have
\begin{align*}
    \frac{\dx}{\dx m} \int_{\R^d} \bar\varrho_m(x) \dx x = 0\end{align*}
    which is equivalent to
 \begin{align*}
0&= \frac{1}{m-1}\int_{\R^d} \prt*{\frac{m-1}{m}(C_m - V(x))_+}^{\frac{2-m}{m-1}} \prt*{\frac{m-1}{m}C'_m{\mathds{1}_{\{C_m > V(x)\}}} + \frac{(C_m - V(x))_+}{m^2}} \dx x \\
    &- \frac{1}{(m-1)^{2}} \int_{\R^d} \prt*{\frac{m-1}{m}(C_m - V(x))_+}^{\frac{1}{m-1}} \ln \prt*{\frac{m-1}{m}(C_m - V(x))_+} \dx x\\
    & = \frac{1}{m-1}\int_{\R^d} \prt*{\frac{m-1}{m}(C_m - V(x))_+}^{\frac{2-m}{m-1}} \prt*{\frac{m-1}{m}C'_m{\mathds{1}_{\{C_m > V(x)\}}} + \frac{(C_m - V(x))_+}{m^2}} \dx x \\
    &- \frac{1}{m-1} \int_{\R^d} \prt*{\frac{m-1}{m}(C_m - V(x))_+}^{\frac{1}{m-1}} \ln \prt*{\frac{m-1}{m}(C_m - V(x))_+}^{\frac{1}{m-1}} \dx x.
    \end{align*}
After rearranging this yields
 \begin{align*} 
 \frac{1}{m-1}\int_{\R^d} (\bar\varrho_m)^{2-m} \prt*{\frac{m-1}{m}C'_m{\mathds{1}_{\{C_m > V(x)\}}} + \frac{(\bar\varrho_m)^{m-1}}{m(m-1)}} \dx x = \frac{1}{m-1} \int_{\R^d} \bar\varrho_m \ln \bar\varrho_m \dx x,
    \end{align*}
and finally  
    \begin{align*}
    \int_{\R^d} \underbrace{(\bar\varrho_m)^{2-m} \frac{C'_m}{m}{{\mathds{1}_{\{C_m > V(x)\}}}}}_{=:h_m(x)}\dx x= \frac{1}{m-1} \int_{\R^d} \bar\varrho_m \ln \bar\varrho_m\dx x - \frac{1}{m(m-1)^2}\int_{\R^d} \bar\varrho_m \dx x.
\end{align*}
This computation implies in particular that 
$$
\partial_{m}\bar \varrho_m(x) = h_m - \frac{1}{m-1}\bar\varrho_m(x) \ln \bar\varrho_m(x) + \frac{1}{m(m-1)^2}\bar\varrho_m(x).
$$

Since $C'_m$ is independent of $x$, for any fixed $m$ the function $h_m$ has a sign which matches the sign of the constant $C'_{m}$. Let us assume, for instance, that $h_m$ is positive. Then, we have
\begin{equation}\label{dmrho}
\begin{split}    \int_{\R^d} |\partial_m \bar\varrho_m| \dx x &= \int_{\R^d} \left|h_m - \frac{1}{m-1}\bar\varrho_m \ln \bar\varrho_m + \frac{1}{m(m-1)^2}\bar\varrho_m  \right| \dx x\\[0.6em]
    &= \int_{\R^d}\partial_m \bar\varrho_m \dx x + 2\int_{\R^d} \left(h_m - \frac{1}{m-1}\bar\varrho_m \ln \bar\varrho_m  + \frac{1}{m(m-1)^2}\bar\varrho_m \right)_- \dx x \\[0.6em]
    &\leq 2 \int_{\R^d} \left(- \frac{1}{m-1}\bar\varrho_m \ln \bar\varrho_m + \frac{1}{m(m-1)^2}\bar\varrho_m  \right)_- \dx x \\[0.6em]
    &\leq 2 \int_{\R^d} \left|\frac{1}{m-1}\bar\varrho_m \ln \bar\varrho_m \right| \dx x + 2 \frac{1}{m(m-1)^2}\int_{\R^d}\bar\varrho_m \dx x,
\end{split}
\end{equation}
where we have used the fact $|f| = f + 2 (f)_-$ and $\int_{\R^d}\partial_m \bar\varrho_m \dx x=0$. The same works if $h_m$ is negative using $|f|=-f +2 (f)_+$.

We may use a corollary of the co-area formula (see \cite[Section 3.4.4, Proposition 3]{EvaGar}) to compute the $L^1$-norm of $\bar\varrho_m \ln \bar\varrho_m$ as follows
\begin{align*}
 &\int_{\R^d} |\bar\varrho_m(x) \ln \bar\varrho_m(x)| \dx x \\
 &\:=  \frac{1}{m-1}\int_0^{C_m}\!\!\! \int_{\{V(x)=s\}}\!\! \prt*{\frac{m-1}{m}(C_m -s)_+}^{\frac{1}{m-1}} \left|\ln \prt*{\frac{m-1}{m}(C_m -s)_+} \right|\frac{1}{|\nabla V(x)|} \dx{\sigma(x)} \dx s,
\end{align*}
where we use the shorthand notation $\sigma:=\mathcal{H}^{d-1}$ (the $(d-1)$-dimensional Hausdorff measure).
Now we use the assumptions on the potential $V$, in particular to handle the potentially vanishing term $\frac{1}{|\nabla V(x)|}$. Since $\alpha I \leq D^2 V\leq A I$, $V$ is convex, $|\nabla V|\geq \alpha |x|$, and $\{V(x)=s\}\subset B_{\sqrt{{s}/\alpha}}$. Therefore $\mathrm{Per}\{V(x)=s\} \leq C s^{(d-1)/2}$. Moreover, the upper bound on the Hessian implies $V\leq \frac A 2 |x|^2$, thus
\begin{align}
  & \frac{1}{m-1}\int_0^{C_m} \int_{\{V(x)=s\}} \prt*{\frac{m-1}{m}(C_m -{s})_+}^{\frac{1}{m-1}} \left|\ln \prt*{\frac{m-1}{m}(C_m -{s})_+} \right|\frac{1}{|\nabla V(x)|} \dx{\sigma(x)} \dx {s} \\[0.6em]
     &\:\leq  \frac{1}{m-1}\int_0^{C_m}  \prt*{\frac{m-1}{m}(C_m -{s})_+}^{\frac{1}{m-1}} \left|\ln \prt*{\frac{m-1}{m}(C_m -{s})_+} \right| \int_{\{V(x)={s}\}} \frac{1}{\alpha |x|} \dx{\sigma(x)}\dx {s}\\[0.6em]
     &\:\leq  \frac{1}{m-1}\int_0^{C_m}  \prt*{\frac{m-1}{m}(C_m -{s})}^{\frac{1}{m-1}} \left|\ln \prt*{\frac{m-1}{m}(C_m -{s})} \right|\int_{\{V(x)={s}\}} \frac{\sqrt{A }}{\alpha\sqrt{2{V}}}\dx{\sigma(x)}\dx {s}\\[0.6em]
     &\:\leq \frac{C}{m-1}\int_0^{C_m}  \prt*{\frac{m-1}{m}(C_m -{s})}^{\frac{1}{m-1}} \left|\ln \prt*{\frac{m-1}{m}(C_m -{s})} \right| {s}^{(d-2)/2} {\dx {s}}\\[0.6em]
     &\:\leq  \frac{C}{m-1}.
\end{align}
{The last inequality here is the consequence of two facts: first, the constants $C_{m}$ are uniformly bounded with respect to $m$ and second the uniform boundedness of the integral. Indeed, since $s\mapsto s^{(d-2)/2}$ is uniformly bounded on $[0,C_{m}]$, we have
\begin{align*}
\frac{1}{m-1}\int_0^{C_m} & \prt*{\frac{m-1}{m}(C_m -s)}^{\frac{1}{m-1}} \left|\ln \prt*{\frac{m-1}{m}(C_m -s)} \right| s^{(d-2)/2} \dx s\\
& \le \frac{C}{m-1}\int_0^{C_m} \prt*{\frac{m-1}{m}(C_m -s)}^{\frac{1}{m-1}} \left|\ln \prt*{\frac{m-1}{m}(C_m -s)} \right| \dx s\\
& = C\int_0^{C_m} \prt*{\frac{m-1}{m}s}^{\frac{1}{m-1}} \left|\ln \prt*{\frac{m-1}{m}s}^{\frac{1}{m-1}} \right| \dx s\\
& = C\int_0^{\prt*{\frac{m-1}{m}C_{m}}^{\frac{1}{m-1}}} r |\ln r| \frac{m}{m-1} r^{m-1}\dx r,
\end{align*}
where we have used the change of variable formula $r:=\prt*{\frac{m-1}{m}s}^{\frac{1}{m-1}}.$ Again, as $\prt*{\frac{m-1}{m}C_{m}}^{\frac{1}{m-1}}$ is uniformly bounded with respect to $m$ and the function $r\mapsto r |\ln r|$ is uniformly bounded on the domain of integration, by increasing potentially the value of $C$, we find that the previous integral is bounded above by
\begin{align*}
C\frac{m}{m-1} \int_0^{\prt*{\frac{m-1}{m}C_{m}}^{\frac{1}{m-1}}}  r^{m-1}\dx r = C\frac{m}{(m-1)m} \prt*{\frac{m-1}{m}C_{m}}^{\frac{m}{m-1}},
\end{align*}
and so the claim follows.
}

Finally, coming back to \eqref{dmrho}, we have
\begin{equation}\label{rate l1}
    \|\partial_m \bar\varrho_m\|_{L^1}\leq \frac{C}{(m-1)^2}{+ \frac{2}{m(m-1)^2} \leq \frac{C}{(m-1)^2}},
\end{equation}
{which concludes the first inequality in the statement of this lemma.

\medskip

For the second inequality,} we compute the derivative with respect to $m$ of the square of the $W_{2}$ distance. {This is a direct consequence of the first variation formula presented in \cite[Proposition 7.17]{OTAM}, which implies that}
\begin{equation}
    \label{dm W2}
    \begin{aligned}
    \partial_m W_2^2(\bar\varrho_m, \bar\varrho_\infty) &= \partial_m \prt*{\int_{\R^d} \varphi \bar\varrho_m(x) \dx x +\int_{\R^d} \psi \bar\varrho_\infty(x) \dx x }\\[0.3em]
    &= \int_{\R^d} \varphi \partial_m \bar\varrho_m(x) \dx x\\[0.3em]
    &\leq \|\varphi\|_{L^\infty} \|\partial_m \bar\varrho_m\|_{L^1}.
\end{aligned}
\end{equation}
Here, $(\varphi,\psi)$ is a pair of Kantorovich potentials in the optimal transport of $\bar\varrho_m$ onto $\bar\varrho_\infty$.  Let us notice that Kantorovich potentials are unique up to additive constants. Because of this, without loss of generality one might always assume that $\varphi(0)=0$. Thus, 
$$|\varphi(x)|=|\varphi(x)-\varphi(0)|\le \|\nabla\varphi\|_{L^{\infty}}|x|.$$ Since ${\rm{spt}}(\bar\varrho_{m})$ and ${\rm{spt}}(\bar\varrho_{\infty})$ are uniformly bounded, for all $m>1$, we find that $\|\nabla\varphi\|_{L^{\infty}}$ and $|x|$ are uniformly bounded for all $m$ and all $x\in{\rm{spt}}(\bar\varrho_{m})$. Therefore, $\|\varphi\|_{L^{\infty}}$ is uniformly bounded.
From \eqref{dm W2} we deduce
\begin{equation}
    \left| \partial_m W_2^2(\bar\varrho_m, \bar\varrho_\infty)\right| \leq \frac{C}{(m-1)^2},
\end{equation}
{which concludes the proof of the lemma.}
\end{proof}

{
\begin{corollary}
The second inequality in the statement of Lemma \ref{lem:sec5} readily implies that for these stationary states we have
\begin{equation}
  W_2(\bar\varrho_m, \bar\varrho_\infty) \leq \frac{C}{\sqrt m}.
\end{equation}
Indeed, this is since 
\begin{align*}
W_2^{2}(\bar\varrho_m, \bar\varrho_\infty) = - \int_{m}^{\infty} \partial_{s} W_2^{2}(\bar\varrho_s, \bar\varrho_\infty)\dx s \le \int_{m}^{\infty} \frac{C}{s^{2}}\dx s = \frac{C}{m}.
\end{align*}
\end{corollary}
}
%
%

Thanks to the bound on $\partial_m \varrho_{m}$ established in 
{Lemma \ref{lem:sec5}} it is immediate to deduce a rate of convergence in the $L^1$ norm. 
\begin{proof}[Proof of Theorem~\ref{thm: L1 s.s.}]
We compute
    \begin{equation}
    |\partial_m \|\bar\varrho_m - \bar\varrho_\infty\|_{L^1}| = \left| \partial_m \int_{\R^d} |\bar\varrho_m -\bar\varrho_\infty| \dx x \right| \leq \int_{\R^d}|\partial_m \bar\varrho_m| \dx x,
\end{equation}
and thus from \eqref{rate l1} we have
\begin{equation}
    |\partial_m\|\bar\varrho_m - \bar\varrho_\infty\|_{L^1}|\leq \frac{C}{m^2},
    \end{equation}
which implies
\begin{equation}
    \|\bar\varrho_m - \bar\varrho_\infty\|_{L^1}\leq \frac{C}{m}.
    \end{equation}
\end{proof}
It is worth stressing that finding a rate in $L^1$ for general non-stationary solutions, even locally in time, is non-trivial. As shown in \cite{DDP2021}, through interpolation with the $BV$-norm -- which requires stronger assumptions on the potential -- one can infer rates in some $L^p$ space for $p>1$, but in general rates obtained by interpolation are far from being optimal.

\subsection{Improved $W_2$ rate}

We now prove that stationary solutions can converge faster than $1/\sqrt{m}$ in the $W_2$ distance as $m\to\infty$. This will be the case whenever we can guarantee that the support of the stationary states at level $m$ is contained in the support of the stationary state $\bar\varrho_\infty$. Clearly, this property does not always hold -- for instance in the case of $V$ very ``flat''. However, if the second derivative of the potential is large enough, there exist examples for which it can be shown 
$$\overline{\{x: \ \bar\varrho_m (x)>0\}} \subset \overline{\{x: \ \bar\varrho_\infty(x)>0\}}$$ 
for all $m\gg 1$. We will see later in which cases this holds.

\bigskip

\noindent
Now we prove {Theorem}~\ref{prop: s.s. improved}, namely that the rate in $W_2$ can be improved for stationary solutions satisfying ${\rm{spt}}({\bar\varrho_m})\subset{\rm{spt}}({\bar\varrho_\infty})$.  First of all, we recall a well-known property of the $2$-Wasserstein distance that holds when the measures are defined on a bounded convex open subset of $\R^d$, and one of the two measures is absolutely continuous with density bounded away from zero, see \cite{BJR2007} for the original and more general result. Since $\bar\varrho_\infty$ is the characteristic function of a ball an analogous result holds for $\bar\varrho_\infty, \bar\varrho_m$ if we impose an inclusion constraint on the supports of $\bar\varrho_m$, as stated in the following lemma. 
\begin{lemma}\label{lemma: linfty} 
    Let $\bar\varrho_m, \bar\varrho_\infty$ be defined by \eqref{eq: s.s. m} and \eqref{eq: s.s. infty}, and such that ${\rm spt}({\bar\varrho_m}) \subset {\rm{spt}}({\bar\varrho_\infty})$. Let $(\varphi,\psi)$ be the pair of Kantorovich potential between $\bar\varrho_m$ and $\bar\varrho_\infty$. Then, it holds
    \begin{equation}
       \|\nabla \varphi\|_{L^\infty(\R^d)} = \|\nabla \psi\|_{L^\infty({\rm{spt}}(\bar\varrho_\infty))} \leq C(d)W_2(\bar\varrho_m, \bar\varrho_\infty)^{\frac{2}{d+2}}.
    \end{equation}
\end{lemma}
\begin{proof}
{The proof of this result is a direct consequence of Lemma~\ref{lemma 1}, and so, we omit it.}
\end{proof}

\begin{proof}[Proof of {Theorem}~\ref{prop: s.s. improved}]

The argument relies on refining the estimate used to prove the convergence rate of the stationary states in the proof of Theorem~\ref{thm: main}. For the sake of simplicity, let us denote $K:={\rm spt}(\varrho_\infty)$. We recall that $(\varphi,\psi)$ is the pair of Kantorovich potentials for the transport from $\bar\varrho_m$ to $\bar\varrho_\infty$. Analogously to \eqref{dm W2} we compute
\begin{equation}\label{dm W2 new}  
\begin{aligned}
  \partial_m W_2^2(\bar\varrho_m, \bar\varrho_\infty)&= \int_{\rm{spt}{(\bar\varrho_m)}} \varphi (x) \partial_m \bar\varrho_m(x) \dx x\\[0.5em]
  &= \int_{K} \varphi (x) \partial_m \bar\varrho_m(x) \dx x\\[0.5em]
  &\leq \|\varphi\|_{L^\infty(K)} \|\partial_m \bar\varrho_m\|_{L^1(K)}\\[0.5em]
  &=\|\psi\|_{L^\infty(K)} \|\partial_m \bar\varrho_m\|_{L^1(K)}\\[0.5em]
  &\leq \frac{1}{m^2} \|\psi\|_{L^\infty(K)}.
\end{aligned}
\end{equation}
We used here the relation between $\varphi$ and $\psi $ in order to estimate $\|\varphi\|_{L^\infty}$ with $\|\psi\|_{L^\infty}$. We have $\varphi(x)=\inf_y \frac12|x-y|^2-\psi(y)$. From this we infer $\varphi(x)\geq \inf_y-\psi(y)\geq -\|\psi\|_{L^\infty}$ and, using $y=x$, $\varphi(x)\leq-\psi(y)\leq \|\psi\|_{L^\infty}$. It is important here to assume $\mathrm{spt}(\bar\varrho_m)\subset\mathrm{spt} (\bar\varrho_\infty)$ in order to take $y=x$. Then, instead of simply controlling $\|\psi\|_\infty$ by a uniform constant we use the Sobolev inequality to find
\begin{equation}
    \|\psi\|_{L^\infty(K)}  \leq C \left(\int_{K} |\nabla \psi(x)|^{q} \dx x\right)^{1/{q}},
\end{equation}
with ${q}>d$. Using the definition of $W_2(\bar\varrho_m,\bar\varrho_\infty)$, we have
\begin{align}
    \|\psi\|_{L^\infty(K)} &\leq {C}\|\nabla \psi\|_\infty^{\frac{{q}-2}{{q}}} \left(\int_{K} |\nabla \psi(x)|^2 \dx x\right)^{1/{q}}\\[0.5em]
     &\leq {C}\|\nabla \psi\|_\infty^{\frac{{q}-2}{{q}}} W_2(\bar\varrho_m,\bar\varrho_\infty)^{\frac 2{q}}\\[0.5em]
     &\leq {C}W_2(\bar\varrho_m,\bar\varrho_\infty)^{\frac{2({q}-2)}{{q}(d+2)} +\frac 2{q}},
\end{align}
where in the last inequality we used Lemma~\ref{lemma: linfty}.
Coming back to \eqref{dm W2 new}, we have
\begin{equation}
              \partial_m W_2^2(\bar\varrho_m, \bar\varrho_\infty)\leq \frac{C}{m^2} W_2^{2\kappa}{(\bar\varrho_m, \bar\varrho_\infty)},   
\end{equation}
with $\kappa = \frac{d+{q}}{{q}(d+2)}$, which {is an exponent in $(0,1)$ since $q>d$. Using similar arguments, we deduce also that 
\begin{equation}
              \partial_m W_2^2(\bar\varrho_m, \bar\varrho_\infty)\ge - \frac{C}{m^2} W_2^{2\kappa}{(\bar\varrho_m, \bar\varrho_\infty)}.
\end{equation}
This latter inequality
}
implies that 
\begin{equation}
    W_2(\bar\varrho_m,\bar\varrho_\infty) \leq \frac{C}{m^{1/2(1-\kappa)}},
\end{equation}
{which is a direct consequence of a Gr\"onwall type argument. Indeed, by setting $f(t):=W_2^2(\bar\varrho_t, \bar\varrho_\infty)$, we have the differential inequality $f'(t) \ge -\frac{C}{t^{2}}f(t)^{\kappa}$ on the interval $(m,+\infty)$, and $\lim_{t\to+\infty}f(t) = 0$. This gives the desired bound}
and this concludes the proof.
\end{proof}

We now want to show examples where the inclusion $\mathrm{spt}(\bar \varrho_m)\subset\mathrm{spt} (\bar \varrho_\infty)$ holds for large values of $m$.
Let the potential be defined as the simple quadratic function
\begin{equation}
    V(x)= A|x|^2,
\end{equation}
and let $A$ be a positive constant such that 
\begin{equation}\label{assum: A}
   A> \omega_d^{2/d}\exp\left(- {d}\int_0^1 \log\prt*{1- s^2} s^{d-1}\dx s\right),
   \end{equation}
where $\omega_d$ denotes the volume of the unit ball in $\R^d$. The right-hand side in the inequality above is a dimension-dependent constant.

\medskip

{\textbf{ Claim.} \it Under assumption \eqref{assum: A}, $\mathrm{spt}(\bar \varrho_m)\subset\mathrm{spt} (\bar \varrho_\infty)$, for $m\gg 1.$}

\smallskip
\noindent
{\it Proof of Claim.} Since $\bar\varrho_m$ is defined as in \eqref{eq: s.s. m}, an easy computation gives the value of $C_m$
\begin{equation}
\begin{aligned}
    \int_{\R^d}\prt*{{\frac{m-1}{m}}\prt*{C_m - A |x|^2}_+}^{\frac{1}{m-1}} \dx x&= 1,\\[0.6em]
  \prt*{\frac{m-1}{m} C_m}^{\frac{1}{m-1}} \prt*{\frac{C_m}{A}}^{\frac  d 2}\int_{\{1>|x|^2\}}\prt*{1- |x|^2}^{\frac{1}{m-1}} \dx x&= 1,\\[0.6em]
  \omega_d\prt*{\frac{m-1}{m} C_m}^{\frac{1}{m-1}} \prt*{\frac{C_m}{A}}^{\frac  d 2}{d}\int_0^1\prt*{1- s^2}^{\frac{1}{m-1}} s^{d-1}\dx s&= 1.
\end{aligned}
\end{equation} 
By definition \eqref{eq: s.s. m}, the support of $\bar \varrho_m$ is the ball $B(0,\sqrt{C_m/A})$, while the support of $\bar \varrho_\infty$ is the ball $B(0,R_0)$ where $\omega_dR_0^d=1$, \ie \ $R_0=\omega_d^{-1/d}$. We write $C_m/A=R^2$ and we want to guarantee $R\leq R_0$. 
From the above equality, we have
$$R_0^{-d}  \prt*{\frac{m-1}{m}}^{\frac{1}{m-1}} (AR^2)^{\frac{1}{m-1}}R^d{d}\int_0^1\prt*{1- s^2}^{\frac{1}{m-1}} s^{d-1}\dx s= 1,$$
and we suppose by contradiction $R>R_0$, which implies
$$  \prt*{\frac{m-1}{m}}^{\frac{1}{m-1}} (AR_0^2)^{\frac{1}{m-1}}{d}\int_0^1\prt*{1- s^2}^{\frac{1}{m-1}} s^{d-1}\dx s\leq  1.$$
Then, we take the logarithm of the above inequality and we use the Jensen inequality 
$$\log \left({d}\int_0^1 f(s)s^{d-1}\dx s\right)\geq{d}\int_0^1\log f(s) s^{d-1}\dx s,$$ 
which holds since $\log$ is concave and ${d}\int_0^1 s^{d-1}\dx s=1$. Thus, we find
$$  \frac{1}{m-1}\log\left(\frac{m-1}{m}\right)+ \frac{1}{m-1}\log (AR_0^2)+{d}\int_0^1 \frac{1}{m-1}\log\prt*{1- s^2} s^{d-1}\dx s\leq  0.$$
Note that ${d}\int_0^1 \frac{1}{m-1}\log\prt*{1- s^2} s^{d-1}\dx s$ is a finite value, since $s\mapsto\log(1-s^{2})$ is integrable and $s\mapsto s^{d-1}$ is bounded on $[0,1]$. We then obtain, by multiplying the previous inequality by $(m-1),$
$$  \log\prt*{\frac{m-1}{m}}+ \log (AR_0^2)+{d}\int_0^1 \log\prt*{1- s^2} s^{d-1}\dx s\leq  0.$$
Taking the limit $m\to\infty$, since the first term in the right-hand side tends to $0$, we obtain the condition 
$$A\leq R_0^{-2}\exp\left(-{d}\int_0^1 \log\prt*{1- s^2} s^{d-1}\dx s\right),$$
which is a contradiction with \eqref{assum: A}.

\section*{Acknowledgments}
This project was supported by the LABEX MILYON (ANR-10-LABX-0070) of Universit\'e de Lyon, within the program ``Investissements d'Avenir'' (ANR-11-IDEX-0007) operated by the French National Research Agency (ANR), and by the European Union via the ERC AdG 101054420 EYAWKAJKOS project.
N.D. would like to acknowledge the hospitality of Durham University during her research stay. A.R.M acknowledges the partial support provided by the EPSRC via the NIA with grant number EP/X020320/1 and by the King Abdullah University of Science and Technology Research Funding (KRF) under Award No. ORA-2021-CRG10-4674.2. A.R.M would also like to acknowledge the hospitality of the Institut Camille Jordan, when the idea for this project was conceived.

\section*{Data availability statement}
No data was generated for the purposes of this research.

\section*{Declarations} The authors declare that there are no conflicts or competing interests.

\bibliographystyle{abbrv}
\bibliography{literature}
\end{document}